\documentclass{amsart}
\usepackage{amsmath,amscd,amsfonts,amssymb}
\usepackage{url}
\usepackage{xcolor}

\usepackage{color}
\definecolor{brown}{rgb}{.6,0,0}
\usepackage[
        colorlinks,
        linkcolor=brown, filecolor=brown,  citecolor=brown, 
        pagebackref,
        pdfauthor={},
  pdftitle={},
]{hyperref}

\usepackage{tikz}
\usepackage{verbatim}
\usetikzlibrary{trees,snakes}
\tikzset{fontscale/.style = {font=\relsize{#1}}}

\newcommand{\Z}{\mathbb{Z}} %
\newcommand{\ZZ}{\mathbb{Z}} %
\newcommand{\Q}{\mathbb{Q}} %
\newcommand{\QQ}{\mathbb{Q}} %
\newcommand{\QQbar}{\overline{\mathbb{Q}}} %
\newcommand{\R}{\mathbb{R}} %
\newcommand{\C}{\mathbb{C}} %
\newcommand{\CC}{\mathbb{C}} %

\newcommand{\p}{\mathfrak{p}} %
\newcommand{\PP}{\mathfrak{P}} %

\newcommand{\Cl}{\operatorname{Cl}}
\DeclareMathOperator{\coker}{coker}

\newcommand{\OO}{\mathcal{O}}

\newcommand{\ag}{\mathfrak{a}} %
\newcommand{\pg}{\mathfrak{p}} %
\newcommand{\cond}{\mathfrak{f}} %
\newcommand{\condt}{\cond_{\rm tame}}
\newcommand{\condw}{\cond_{\rm wild}}
\newcommand{\Nm}{\mathrm{N}}

\newcommand{\Res}{\operatorname{Res}}

\newcommand{\norm}{\mathrm{N}}

\newcommand{\Reg}{\mathrm{Reg}}
\newcommand{\Hsg}{\mathcal{H}}
\newcommand{\order}{\mathcal{O}}
\newcommand{\rela}{\mathcal{R}}
\newcommand{\sto}{\leftarrow}
\newcommand{\clcomp}{\mathcal{C}}
\newcommand{\hR}{\mathrm{HR}}

\newcommand{\F}{\mathbb{F}} %
\newcommand{\FF}{\mathbb{F}} %
\newcommand{\triv}{\mathbf{1}}
\DeclareMathOperator{\Ind}{Ind}
\DeclareMathOperator{\SL}{SL}
\DeclareMathOperator{\Id}{Id}
\DeclareMathOperator{\rad}{rad}

\DeclareMathOperator{\Hom}{Hom}
\DeclareMathOperator{\Ann}{Ann}

\DeclareMathOperator{\Gal}{Gal}
\DeclareMathOperator{\mat}{M}
\DeclareMathOperator{\Aut}{Aut}
\DeclareMathOperator{\Logemb}{\mathcal{L}}
\DeclareMathOperator{\height}{h}

\theoremstyle{definition}
\newtheorem{introtheorem}{Theorem}[section]
\newtheorem{introprop}[introtheorem]{Proposition}

\newtheorem{theorem}{Theorem}[section]

\newtheorem{definition}[theorem]{Definition}
\newtheorem{algorithm}[theorem]{Algorithm}
\newtheorem{lemma}[theorem]{Lemma}
\newtheorem{proposition}[theorem]{Proposition}
\newtheorem{corollary}[theorem]{Corollary}

\newtheorem{example}[theorem]{Example}
\newtheorem{remark}[theorem]{Remark}

\newtheorem*{theorem*}{Theorem}
\newtheorem*{proposition*}{Proposition}
\newtheorem*{notation*}{Notation}

\newcommand{\noell}{p}

\usepackage[textsize=normal]{todonotes}

\makeatletter
\providecommand\@dotsep{5}
\def\listtodoname{List of Todos}
\def\listoftodos{\@starttoc{tdo}\listtodoname}
\makeatother

\makeatletter
\@namedef{subjclassname@2020}{%
  \textup{2020} Mathematics Subject Classification}
\makeatother

\begin{document}

\title{Norm relations and computational problems in number fields}

\author{Jean-Fran\c{c}ois Biasse}
\address{Department of Mathematics and Statistics\\
University of South Florida, 4202 East Fowler Ave, CMC342, Tampa, FL 33620-5700, USA}
\email{biasse@usf.edu}
\author{Claus Fieker}
\address{Fachbereich Mathematik, Technische Universitat Kaiserslautern, 67663 Kaiserslautern, Germany}
\email{fieker@mathematik.uni-kl.de}
\author{Tommy Hofmann}
\address{Department Mathematik, Universität Siegen, Postfach, 57068 Siegen, Germany}
\email{tommy.hofmann@uni-siegen.de}
\author{Aurel Page}
\address{INRIA, Univ.~Bordeaux, CNRS, IMB, UMR 5251, F-33400 Talence, France}
\email{aurel.page@inria.fr}

\date{\today}
\subjclass[2010]{Primary: 11Y16, 20C05, 11R32; Secondary 11R29, 11R04, 11Y40, 11R18, 11R27}

\maketitle

\begin{abstract}
  For a finite group $G$, we introduce a generalization of norm relations in the group algebra $\Q[G]$.
  We give necessary and sufficient criteria for the existence of such relations
  and apply them to obtain relations between the arithmetic invariants of the subfields
  of a normal extension of algebraic number fields with Galois group $G$.
  On the algorithmic side this leads to subfield based algorithms for computing rings of integers,
  $S$-unit groups and class groups.
  For the $S$-unit group computation this yields a polynomial time
  reduction to the corresponding problem in subfields.
  We compute class groups of large number fields under GRH, and new
  unconditional values of class numbers of cyclotomic fields.
\end{abstract}

\section{Introduction}

Let $K/F$ be a normal extension of number fields with Galois group $G$.
Since the beginning of algebraic number theory, the interaction and relations
between the arithmetic invariants of $K$ and its subfields has always been an
important topic.
For example it was already observed by Dirichlet~\cite{Dirichlet1942}, and later generalized by Walter~\cite{Walter1979b},
that for biquadratic fields, that is, $F = \Q$ and $G = C_2 \times C_2$, the class numbers
of $K$ and its three nontrivial subfields $K_1,K_2,K_3$ satisfy the class number formula
$h(K) = 2^i \cdot h(K_1)h(K_2)h(K_3)$,
where $i \in \Z$ depends on the index of certain unit groups.
Thus, once the class numbers of the subfields are known, one can compute the
class number of $K$ up to a power of $2$.

In the 1950s, Brauer~\cite{Brauer1951} and Kuroda~\cite{Kuroda1950} laid the foundation
for a systematic study of such class number formulae by connecting them to
character theoretic properties of $G$. More precisely, for a subgroup $H \leq
G$ we denote by $\Ind_{G/H}(\triv_H)$ the permutation character of $G$ induced by the
trivial character of $H$. For a relation of the form
$\sum_{H \leq G} a_H \Ind_{G/H}(\triv_H) = 0$ 
with $a_H \in \Z$, Brauer proved a corresponding relation between zeta functions and arithmetic invariants of the fixed fields $K^H$ (see also~\cite[Theorem 73]{Frohlich1971}).
In connection with class number formulae, the existence of such relations has also been studied from a computational point of view
by Bosma and de~Smit~\cite{Bosma2001}.

A related, more group theoretic notion, is that of a relation of norms of subgroups.
For a subgroup $H \leq G$ denote by $N_H = \sum_{h \in H} h \in \Q[G]$ the corresponding norm as an element of the rational
group algebra. Then one considers equalities of the form
\begin{equation}\label{eq:classicalrel}
  0 = \sum_{H \leq G} a_H N_H
\end{equation}
in $\Q[G]$ with $a_H \in \Z$.
On the number theoretic side this implies, and is equivalent to, $1 = \prod_{H \leq G} \Nm_{K/K^H}(x)^{a_H}$ for all $x \in K^\times$ (see~\cite{Artin1948}).
The correspondence between relations of characters and norms was already
observed by Walter~\cite{Walter1979b}, who used them to derive a simple proof of Kuroda's class number formula.
A group theoretic study of the lattice of relations between norms was done by Rehm~\cite{Rehm1975}.
Under the name~idempotent relation, various arithmetic and geometric applications
have been given by Kani--Rosen, Park and Yu in~\cite{Kani1989, Kani1994, Park1990,
Park1996, Yu2003}. A connection with Arakelov class groups was described by
Kontogeorgis in~\cite{Kontogeorgis2008}.

Although relations between permutation characters and norms have played a significant role in connecting invariants of $K$ and its subfields,
both notions have not seen a systematic use in computational algebraic number theory, for example, in the
computation of the class group. Until recently, the use of subfields in algorithmic number theory had been restricted
to ad-hoc tricks and heuristic observations.
Recent work of Bauch, Bernstein, de Valence, Lange and van
Vredendaal~\cite{BBVLV} describes
how to reduce the computation of principal ideal generators in multiquadratic fields (that is, $G = C_2^n$) to quadratic subfields, thus for the
first time improving (both in theory and practice) upon classical algorithms by exploiting subfields.
This was then generalized to the computation of $S$-units by Biasse and van Vredendaal~\cite{Biasse2019} and to multicubic fields (that is, $G = C_3^n$) by
Lesavourey, Plantard and Susilo~\cite{LPS}.

The aim of the present paper is to extend these ideas to a larger range of
computational problems and to classify those groups $G$ where these improvements
apply.
To this end, we consider relations of the form
\begin{equation}\label{eq:introrel}
  d = \sum_{i=1}^\ell a_i N_{H_i} b_i
\end{equation}
in $\Q[G]$ with $d \in \Z_{>0}$, $H_i \leq G$ nontrivial and $a_i, b_i \in \Z[G]$. We refer to those relations as \textit{norm relations}.
They generalize the classical relations (\ref{eq:classicalrel}) where the coefficient of
the trivial group is nonzero, which are exactly the relations one needs to
determine invariants of $K$ from those of its subfields.
Although our systematic treatment is new, these norm relations have been used in
an ad-hoc way for $G = C_2 \times C_2$ by Wada~\cite{Wada1966}, Bauch,
Bernstein, de Valence, Lange and van Vredendaal~\cite{BBVLV} and Biasse and
van~Vredendaal~\cite{Biasse2019}, and for $G = C_3 \times C_3$ by
Parry~\cite{Parry1977} and Lesavourey, Plantard and Susilo~\cite{LPS}.
We give a systematic study of these
relations; we link them to fixed point free unitary representations and, from a
theorem of Wolf's~\cite{Wolf1972,Wall2013}, we obtain the following
classification of groups admitting a norm relation (see~Theorem~\ref{thm:genrel}).
Interestingly, the relevant condition is exactly the same in the problem coming
from geometry and topology (existence of space-forms) as in ours, but the good
and bad cases are reversed, as the groups that do not admit a norm relation are
exactly the ones that provide an example of a space-form.

\begin{introtheorem}
  The group $G$ admits a norm relation if and only if $G$ contains a non-cyclic subgroup of order~$pq$, where~$p$ and~$q$ are
      primes, or a subgroup isomorphic to~$\SL_2(\F_p)$ where~$p=2^{2^k}+1$ is a Fermat
      prime with~$k>1$.
\end{introtheorem}

The existence of a norm relation yields the following connection between the class
group of~$K$ and that of its subfields (see Proposition~\ref{prop:apl_normrel}).
For a subgroup $H \leq G$ and~$M$ a~$\Z[G]$-module we denote by $M^H = \{ m
\in M \ | \ hm = m \text{ for all $h \in H$\}}$ the fixed points of $H$ in $M$.
\begin{introprop}
  Let~$G$ be a finite group that admits a norm relation~(\ref{eq:introrel}), and
  let~$K/F$ be a Galois extension of number fields with Galois group~$G$.
  Then the group~$\Cl(K)\otimes \Z[1/d]$ is isomorphic to a direct summand
    of~$\bigoplus_{i=1}^\ell \Cl(K^{H_i})\otimes \Z[1/d]$,
    and the group~$\Cl(K)/\Cl(K)[d]$ is isomorphic to a subgroup
    of~$\bigoplus_{i=1}^\ell \Cl(K^{H_i})$.
\end{introprop}

Compared with Brauer--Kuroda type
relations~\cite{Brauer1951,Kuroda1950,Bartel2012,Bartel2013} and Mackey functor
type relations~\cite{Boltje}, ours is less precise in that it bounds the class
group or the class number without pinning it down exactly, but it is also partly
stronger in that norm relations are more frequent than Brauer relations and
because it is independent of the coefficients of the relation.

Our algorithmic use of norm relations to leverage information on~$K$ from its subfields
uses the following simple but crucial statement (see~Proposition~\ref{prop:meta}).

\begin{introprop}
  Let~$M$ be a $\Z[G]$-module, and assume we have a norm relation~(\ref{eq:introrel}). Then the quotient 
  \[ M/(a_1 M^{H_1} + \dotsb + a_\ell M^{H_\ell}) \]
  has exponent dividing $d$.
\end{introprop}

In particular, if $M$ is finitely generated, we can use the modules
$M^{H_i}$ of fixed points to approximate $M$ by a finite index subgroup, whose index
divides a power of~$d$.
We show how to leverage this result in the following classical problems from
computational algebraic number theory (see~Section~\ref{sec:algo}):
\begin{enumerate}
  \item
    computation of the ring of integers $\order_K$,
  \item
    computation of $S$-unit groups $\order_{K, S}^\times$,
  \item
    computation of the class group $\Cl(K)$.
\end{enumerate}

Note that these problems, in particular (2) and~(3), are at the core of many algorithmic questions in algebraic
number theory and arithmetic geometry, as well as cryptographic applications. 
We implemented our algorithms for the general case in~\textsc{Hecke}~\cite{Hecke} and a
special algorithm for the abelian case in~\textsc{Pari/GP}~\cite{PARI2}. Using
both implementations, we computed class groups of number fields that
are out of reach of other current techniques.
For example, consider the normal closure of $x^{10} + x^8 - 4x^2 + 4$, which is a $C_2 \times A_5$ extension of $\Q$ of degree $120$ and discriminant~$\approx 10^{161}$. Using the first implementation we show that assuming the generalized Riemann hypothesis (GRH) the class number is $1$.
This computation takes only $6$ hours on a single core machine (see
Example~\ref{ex:hecke}).
Using the second implementation we determine under GRH the structure of the
class group of the cyclotomic field~$K = \Q(\zeta_{6552})$ of degree~$1728$ and
discriminant~$\approx 10^{5258}$, and in particular we obtain that~$h_{6552}^+ =
70695077806080 = 2^{24}\cdot 3^3 \cdot 5 \cdot 7^4 \cdot 13$.
This computation takes only $4$ hours on a single core machine (see
Example~\ref{ex:parigp}).

Our methods are also useful for unconditional determination of class groups. As
an example, we certify some new values of class numbers of cyclotomic fields
(Theorem~\ref{thm:unconditional-cyclo}).

\begin{introtheorem}\label{thm:unconditional-cyclo-intro}
  The class numbers and class groups in Tables~\ref{tab:certified-cyclo}
  and~\ref{tab:certified-cyclo-complete} are correct.
\end{introtheorem}

In order to keep the table small, we did not include fields for
which the class number was already known unconditionally.
We also determined the class group structure of many examples for
which the class number was known~\cite{Miller}, but it
is likely that the class group structure could be determined by other methods,
for instance~\cite{AokiFukuda} or by constructing explicit class fields, so we
did not include them.
According to Miller~\cite{Miller}, the largest conductor for which the class
number of a cyclotomic field has been computed unconditionally was~$420$ prior
to our work; we raise this record to~$2520$.
Note that our methods are not restricted to cyclotomic fields, but these number
fields provide a family of examples to which they often apply and that are of
general interest.
Our proof of Theorem~\ref{thm:unconditional-cyclo-intro} does
not use special properties of cyclotomic fields other than their Galois group;
it would be interesting to combine them with special cyclotomic techniques.

\begin{table}
  \caption{Class numbers of cyclotomic fields $\Q(\zeta_n)$}\label{tab:certified-cyclo}
  $n$ conductor, $\varphi(n)$ degree, $h^+$ plus part of class number, $r_2$
  $2$-rank of class group, $r_3$ $3$-rank of class group, $T_1$
  time for the conditional computation, $T_2$ time to unconditionally certify
  the computation.
  \hspace*{-1.1cm}\begin{tabular}[t]{ccccccc c ccccccc}
    \hline
    $n$ & $\varphi(n)$ & $h^+$ & $r_2$ & $r_3$ & $T_1$ & $T_2$
      & \phantom{a} &
      $n$ & $\varphi(n)$ & $h^+$ & $r_2$ & $r_3$ & $T_1$ & $T_2$ \\
    \hline
    $255$ & $128$ & $1$ & $1$ & $1$ & $1$~min & $3$~h &&
      $624$ & $192$ & $1$ & $3$ & $4$ & $2.5$~min & $28$~min \\
    $272$ & $128$ & $2$ & $4$ & $2$ & $1$~min & $8$~h &&
      $720$ & $192$ & $1$ & $3$ & $4$ & $2.5$~min & $24$~min \\
    $320$ & $128$ & $1$ & $0$ & $2$ & $25$~s & $13$~h &&
      $780$ & $192$ & $1$ & $18$ & $1$ & $6.5$~min & $6.5$~min \\
    $340$ & $128$ & $1$ & $3$ & $0$ & $1$~min & $8$~h &&
      $840$ & $192$ & $1$ & $6$ & $4$ & $6$~min & $2$~min \\
    $408$ & $128$ & $2$ & $5$ & $2$ & $3$~min & $21$~min &&
      $455$ & $288$ & $1$ & $14$ & $3$ & $4$~min & $9$~h \\
    $480$ & $128$ & $1$ & $3$ & $4$ & $43$~s & $4$~s &&
      $585$ & $288$ & $1$ & $7$ & $4$ & $4$~min & $10.5$~h \\
    $273$ & $144$ & $1$ & $9$ & $2$ & $34$~s & $5.5$~min &&
      $728$ & $288$ & $20$ & $17$ & $14$ & $3$~min & $2$~h \\
    $315$ & $144$ & $1$ & $4$ & $2$ & $20$~s & $4.5$~min &&
      $936$ & $288$ & $16$ & $11$ & $11$ & $2.5$~min & $2.5$~h \\
    $364$ & $144$ & $1$ & $6$ & $5$ & $25$~s & $11$~min &&
      $1008$ & $288$ & $16$ & $13$ & $10$ & $2.5$~min & $5.5$~h \\
    $456$ & $144$ & $1$ & $1$ & $3$ & $1.5$~min & $8$~h &&
      $1092$ & $288$ & $1$ & $24$ & $7$ & $3$~min & $1$~h \\
    $468$ & $144$ & $1$ & $3$ & $6$ & $25$~s & $12$~min &&
      $1260$ & $288$ & $1$ & $14$ & $7$ & $2.5$~min & $2$~h \\
    $504$ & $144$ & $4$ & $9$ & $6$ & $16$~s & $2$~s &&
      $1560$ & $384$ & $8$ & $40$ & $5$ & $2$~h & $3.5$~h \\
    $520$ & $192$ & $4$ & $18$ & $3$ & $6.5$~min & $16$~min &&
      $1680$ & $384$ & $1$ & $12$ & $8$ & $1$~h & $8$~h \\
    $560$ & $192$ & $1$ & $3$ & $5$ & $2.5$~min & $18$~min &&
      $2520$ & $576$ & $208$ & $38$ & $15$ & $40$~min & $43$~h \\
    \hline
  \end{tabular}
\end{table}

On the theoretical side, assuming GRH we exhibit a polynomial time reduction to proper
subfields in the presence of a norm relation
(Theorem~\ref{thm:redtosubfields}).

\begin{introtheorem}
  Assume GRH holds. Let $G$ be a finite group and $\Hsg$ a set of subgroups
  of~$G$.
  Assume that there exists a norm relation as in~(\ref{eq:introrel}) where
  all~$H_i$ belong to~$\Hsg$.
  There exists a deterministic polynomial time algorithm that, on input of
  \begin{itemize}
    \item a number field $K$,
    \item an injection $G \to \Aut(K)$,
    \item a finite $G$-stable set $S$ of prime ideals of $K$,
    \item for each $H$ in $\Hsg$, a basis of the group of $S$-units of the subfield
          fixed by $H$,
  \end{itemize}
  returns a $\Z$-basis of the group of  $S$-units of $K$.
\end{introtheorem}

The proof uses an effective version of the Grunwald--Wang theorem
under GRH, which is different from other versions found in the literature (for
instance~\cite{Wang2015}) and
may be of independent interest (Theorem~\ref{thm:gruneff}), and a bound on the
smallest possible value of~$d$ in norm relations (Theorem~\ref{thm:denrelG}).
We also provide an easily checked criterion for the existence of a norm relation
(Proposition~\ref{prop:generel}), and a complete classification of optimal norm
relations in the abelian case (Theorem~\ref{thm:optimalabelianrels}).

In view of the previous theorems, one might ask to which extent norm
relations exhaust all possibilities to exploit subfields to compute $S$-units.
We partially answer this by proving the following converse (see
Proposition~\ref{prop:conv}).

\begin{introprop}
  Let~$K/F$ be a finite normal extension of number fields with Galois group~$G$, and let~$\Hsg$
  be a set of nontrivial subgroups of~$G$.
  Let~$S$ be a finite $G$-stable set of prime ideals of~$K$.
  Assume that at least one of the following holds:
  \begin{itemize}
    \item $F$ is not totally real,
    \item there is a real place of~$F$ that splits completely in~$K$, or
    \item there is a prime ideal~$\pg$ of~$F$ that splits completely in~$K$ and
      such that the primes above~$\pg$ are in~$S$.
  \end{itemize}
  If the $\Z[G]$-submodule of~$\OO_{K,S}^\times$ generated by the~$\OO_{K^H,S}^\times$
  for~$H\in\Hsg$ has finite index, then~$G$ admits a norm relation with respect
  to~$\Hsg$.
\end{introprop}

Note that when using a set~$S$ that is guaranteed to generate the class group
of~$K$ by analytic bounds, the third condition of the proposition is usually
satisfied.

The paper is structured as follows. In Section~\ref{sec:normrels} we recall the
definitions of the classical Brauer relations and relations between norms and
introduce our notion of norm relations.
We then go on to prove the necessary and sufficient conditions for
the existence of these relations. We also investigate arithmetic properties
of such relations, which play an important role in the number theoretic
applications. We describe these applications
in Section~\ref{sec:applications}, where we also explain the
consequences of the existence of a norm relation for the invariants of number
fields. We then exploit these properties from an algorithmic point of view
in Section~\ref{sec:algo}.
Finally, in Section~\ref{sec:examples} we give various examples of
computations of class groups of abelian and non-abelian number fields.

\subsection*{Acknowledgment}
J.-F.~Biasse was supported by National Science Foundation grant 183980,
National Science Foundation grant 1846166, National Institute of Standards and Technology grant 60NANB17D184, CyberFlorida Collaborative Seed Grant Program and
CyberFlorida Capacity Building Program.
C.~Fieker was supported by Deutsche Forschungsgemeinschaft —
Project-ID 286237555 – TRR 195, Project-ID 460135501 - NFDI 29/1 ``MaRDI -- Mathematische
Forschungsdateninitiative'' and the state of Rheinland-Pfalz via the
Forschungsinitiative and the ``SymbTools'' project. T.~Hofmann was supported by Deutsche Forschungsgemeinschaft —
Project-ID 286237555 – TRR 195; and Project-ID 539387714.
A.~Page was supported by the ANR grant Ciao ANR-19-CE48-0008.
He would like to thank K.~Belabas and B.~Allombert for improving some
\textsc{Pari/GP} functionalities that were useful for this project and for sharing
insight on classical computational problems in number fields, A.~Bartel for
rich discussions about representation theory, and P. Kirchner for a suggestion
to improve Theorem~\ref{thm:redtosubfields}. He also owes some inspiration
for this project to a talk given by T.~Fukuda about computations of class
groups of abelian fields.
The authors also wish to thank Gunter Malle for numerous helpful comments, and
an anonymous referee for his thorough and thoughtful report that lead to
improvements of the paper.

\subsection*{Notations}
\hfill
We will use~$x\mapsto \overline{x}$ to denote various canonical projection maps
that should be clear from the context.

Let~$A$ be a ring, and let~$X\subseteq A$ be a subset. We write~$\langle
X\rangle_A = \sum_{x\in X}AxA$ for the two-sided ideal of~$A$ generated by~$X$.

We will denote by~$1$ the trivial group.
Let~$G$ be a finite group and~$R$ a commutative ring.
Let~$M$ be an~$R[G]$-module. We write~$M^G = \{m\in M \mid gm=m \text{ for all
}g\in G\}$ for the $R$-submodule of fixed points under~$G$.
In case $M$ is a $\Z$-module whose operation is expressed using multiplicative notation, e.g., the multiplicative group of a field or the multiplicative group of fractional ideals,
we write the action of $\Z[G]$ on $M$ as powers, by the formula
\[
  x^a = \prod_{g\in G} g(x)^{a_g} \text{ for all }x \in M, \ a = \sum_{g\in
  G}a_g g\in \Z[G].
\]
Note that this is a left action, i.e. it satisfies~$x^{ab} = (x^b)^a$ for
all~$a,b\in \Z[G]$.

Let~$H$ be a subgroup of~$G$, which we write~$H\le G$.
We denote by $N_H = \sum_{h \in H} h \in \ZZ[G]$ the norm element of $H$.
Let~$M$ be an $R[H]$-module. We use the notation~$\Ind_{G/H}(M)$ for the
induction~$R[G]\otimes_{R[H]}M$ of~$M$ to~$G$.
Let~$\chi$ be the character of a~$\C[H]$-module~$M$; we write~$\Ind_{G/H}(\chi)$
for the character of~$\Ind_{G/H}(M)$, and we write~$\Res_{G/H}(\chi)$ for the
restriction of~$\chi$ to~$H$.
Let~$F_1,F_2$ be $\C$-valued class functions on~$G$. We write their inner
product $\frac{1}{|G|}\sum_{g\in
G}F_1(g)\overline{F_2(g)}$ as~$\langle F_1,F_2\rangle_G$.
We denote by $\triv_G \colon G \to \C^\times$ the trivial character, which
satisfies $\triv_G(g) = 1$ for all $g \in G$.
We denote by~$\varphi$ the Euler totient function.

Let~$A$ be a finite abelian group, written additively here.
For a prime number~$p$, if~$A$ has cardinality~$mp^k$ with~$k\ge 0$ and~$m$ not
divisible by~$p$, let~$A_p = A/p^kA$ be the $p$-part of~$A$, and~$A_{p'} = A/mA$
be the coprime-to-$p$ part of~$A$. We have~$A \cong A_p\times A_{p'}$.
For an integer~$d$, denote by~$A[d]$ the $d$-torsion subgroup~$\{a\in A \mid da=0\}$.

\section{Brauer and norm relations of finite groups}\label{sec:normrels}

\subsection{Brauer relations and norm relations}

Let $G$ be a finite group and let~$H\le G$ be a subgroup. We recall a few basic properties
of the norm element~$N_H = \sum_{h\in H}h$:
\begin{itemize}
  \item For all~$h\in H$ we have~$hN_H = N_Hh = N_H$.
  \item For all~$g\in G$ we have~$gN_Hg^{-1} = N_{gHg^{-1}}$.
  \item For every~$\Z[G]$-module~$M$ and~$x\in M$, we have~$N_Hx\in M^H$.
  \item We have~$N_H^2 = |H| \cdot N_H$.
  \item If~$R$ is a commutative ring where~$|H|$ is invertible,
    then~$e = \frac{1}{|H|}N_H\in R[G]$ is an idempotent and for
    every~$R[G]$-module~$M$ we have~$eM = N_HM = M^H$.
\end{itemize}

\begin{definition}
  Let $G$ be a finite group and $\Hsg$ a set of subgroups of $G$.
  \begin{enumerate}
    \item
      A~\emph{Brauer relation} $\rela$ of $G$ with respect to $\Hsg$ is an equality of the form
      \[
        0 = \sum_{H \in \Hsg} a_H \Ind_{G/H}(\triv_H)
      \]
      with~$a_H\in\QQ$, where the equality is as class functions on~$G$.
      We call $\rela$ \emph{useful} if $1 \in \Hsg$ and $a_{1}\neq 0$.
      If $\Hsg$ is the set of all subgroups of $G$, we simply call $\rela$ a Brauer relation.
    \item
      Let~$R$ be a commutative ring. A \emph{norm relation over~$R$ with respect to~$\Hsg$} (or
      simply \emph{norm relation} if~$R=\QQ$) is an equality of the form
      \[
        1 = \sum_{i=1}^\ell a_i N_{H_i} b_i
      \]
      with~$a_i,b_i\in R[G]$ and $H_i \in \Hsg$, $H_i \neq 1$, where the equality holds in~$R[G]$.
    \item
      A \emph{scalar norm relation} $\rela$ of $G$ over $R$~with respect to $\Hsg$ is an equality of the form
      \[
        0 = \sum_{H \in \Hsg} a_H N_H
      \]
      with~$a_H\in R$ and $a_1 \neq 0$, where the equality holds in the group algebra~$R[G]$.
      If $\Hsg$ is the set of all subgroups of $G$, we simply call $\rela$
      a scalar norm relation.
\end{enumerate}
  We always omit~$\Hsg$ from the terminology when~$\Hsg$ is the
  set of all nontrivial subgroups of~$G$.
\end{definition}

\begin{remark}\label{rem:defrels}
  \hfill
  \begin{enumerate}
    \item
      In the literature, the term~\textit{norm relation} is also used
      for relations of the form
      \[ 0 = \sum_{H \leq G} a_H N_H \]
      with $a_H \in \Q$.
      In this regard, we consider and generalize classical norm relations
      with $a_1 \neq 0$. In particular, our norm relations are by definition
      always nonzero.
    \item Brauer relations of finite groups have been completely classified by
      Bartel and Dokchitser (\cite{Bartel2015, Bartel2014}).
    \item\label{item:replace} Let~$\tilde H\le H$ be a subgroup. We have~$N_H = \sum_{h\in H/\tilde H}hN_{\tilde H}$, so
      in a norm relation where~$H$ appears we may always replace it
      by~$\tilde H$ at the cost of increasing the number of terms.
  \end{enumerate}
\end{remark}

\begin{example}
  Let~$p$ be a prime, and let~$G = C_p\times C_p$. Then we have the scalar norm
  relation
  \[
    p = \left(\sum_{C\le G,\, |C|=p}N_C\right) - N_G.
  \]
  Indeed, every nontrivial element of~$G$ has order~$p$, there are~$p+1$
  subgroups of order~$p$, and every nontrivial element is contained in exactly
  one of them.
\end{example}

\begin{example}\label{ex:rel-pq}
  Let~$p$ be a prime, and let~$q$ be a prime dividing~$p-1$.
  Let~$G = C_p \rtimes C_q$ be a nontrivial semidirect product.
  Then we have the scalar norm relation
  \[
    p = N_{C_p} + \left(\sum_{C\le G,\, |C|=q}N_C\right) - N_G.
  \]
  Indeed, every nontrivial element of~$G$ has order~$p$ or~$q$, there is a
  unique subgroup of order~$p$, there are~$p$ subgroups of order~$q$, and every
  nontrivial element is contained in exactly one of them.
\end{example}

\begin{example}
  Let~$G = C_2\times C_2 = \langle \sigma, \tau\rangle$.
  Then we have the norm relation
  \[
    2 = N_{\langle\sigma\rangle} + N_{\langle\tau\rangle} - \sigma
    N_{\langle\sigma\tau\rangle}.
  \]
  This is the relation used by Wada \cite{Wada1966}, Bauch, Bernstein, de~Valence, Lange and van~Vredendaal \cite{BBVLV} as well as by 
  Biasse and van~Vredendaal \cite{Biasse2019}.
\end{example}

\begin{example}
  Let~$G = C_3\times C_3 = \langle u, v\rangle$.
  Then we have the norm relation
  \[
    3 = N_{\langle u\rangle} + N_{\langle v\rangle} + N_{\langle uv\rangle}
    - (u+uv)N_{\langle u^2v\rangle}.
  \]
  This is the relation used by Parry \cite{Parry1977} and by Lesavourey, Plantard and Susilo \cite{LPS}.
\end{example}

\subsection{Existence of relations}

We now discuss the existence of the various relations. We begin by showing that
Brauer and scalar norm relations are in essence the same.

\begin{proposition}\label{prop:brauernormrel}
  Let $G$ be a finite group and $\Hsg$ a set of subgroups of $G$.
  \begin{enumerate}
    \item
      If 
      \[ 0 = \sum_{H \in \Hsg} a_H N_H \]
      is a scalar norm relation with respect to $\Hsg$, then
      \[ 0 = \sum_{H \in \Hsg} a_H \lvert H \rvert \Ind_{G/H}(\triv_H) \]
      is a useful Brauer relation with respect to $\Hsg$.
    \item
      If
      \[ 0= \sum_{H \in \Hsg} a_H \Ind_{G/H}(\triv_H) \]
      is a useful Brauer relation with respect to $\Hsg$ then
      \[
        0 = \sum_{H\in\Hsg} \frac{a_H}{|H|} \sum_{g\in G}N_{gHg^{-1}}
      \]
      is a scalar norm relation;
      if in addition $\Hsg$ is invariant under conjugation, then
      \[
        0 = \sum_{H \in \Hsg} \left(\frac{1}{\lvert H \rvert} \sum_{g \in
        G}a_{gHg^{-1}}\right) N_{H}
        \]
    is a scalar norm relation with respect to $\Hsg$.
\end{enumerate}

\end{proposition}
\begin{proof}
  The statements are implicitly contained in \cite{Walter1979b}. For the sake
  of completeness we include a proof. We will make use of the fact that
  \[ \sum_{g \in G} \Ind_{G/H}(\triv_H)(g) \cdot g = \lvert H \rvert^{-1} \sum_{g \in G} g N_H g^{-1} \]
  for all subgroups $H \leq G$.
  (1): Assume now that $\sum_{H \in \Hsg} a_H N_H = 0$ is a scalar norm relation.
  Then also $\sum_{H \in \Hsg} a_H g N_H g^{-1} = 0$ for all $g \in G$ and summing over all $g \in G$ yields
  \[ 0 = \sum_{g \in G} \sum_{H \in \Hsg} a_H g N_H g^{-1}
       = \sum_{H \in \Hsg} a_H \sum_{g \in G} g N_H g^{-1}
       = \sum_{H \in \Hsg} a_H \lvert H \rvert \lvert H \rvert^{-1} \sum_{g \in G} g N_H g^{-1}. \]
  Hence 
  \[ 0 = \sum_{H \in \Hsg} a_H \lvert H \rvert \sum_{g \in G} \Ind_{G/H}(\triv_H)(g)\cdot g = \sum_{g \in G} \left(\sum_{H \in \Hsg}  a_H \lvert H \rvert  \Ind_{G/H}(\triv_H)(g)\right)\cdot g\]
  in $\QQ[G]$. Thus $\sum_{H \in \Hsg} a_H \lvert H \rvert \Ind_{G/H}(\triv_H) = 0$ is a Brauer relation, which is useful since $a_1 \neq 0$.

  (2): Assume now that $\sum_{H \in \Hsg} a_H \Ind_{G/H}(\triv_H) = 0$ is a useful Brauer relation with respect to $\Hsg$.
  Then from the above computation we conclude that
  \[ 0 = \sum_{g \in G} \sum_{H \in \Hsg} \frac{a_H}{|H|} g N_H g^{-1} = \sum_{H
  \in \Hsg} \sum_{g \in G} \frac{a_H}{|H|} N_{gHg^{-1}}, \]
  which is a scalar norm relation since the coefficient of~$H=1$ is~$a_1|G|\neq
  0$. If~$\Hsg$ is invariant under conjugation, then after reordering this relation
  becomes the claimed scalar norm relation with respect to $\Hsg$.
\end{proof}

We illustrate the second statement by the following example.

\begin{example}
  Consider the symmetric group $G = S_3$ on three letters and the set of subgroups $\Hsg = \{G, C_3, C_2, 1\}$, where $C_2$ is generated by any of the transpositions.
  Then $G$ admits the useful Brauer relation
  \[ 0 = \Ind_{G/1}(\triv_1) + 2 \Ind_{G/G}(\triv_G) - \Ind_{G/C_3}(\triv_{C_3}) - 2\Ind_{G/C_2}(\triv_{C_2}). \] 
  It is easy to see that $G$ does not admit a scalar norm relation with respect to~$\Hsg$.
  The scalar norm relation given by the proposition is
  \[
    0 = 6N_1 + 2N_G - 2N_{C_3} - 2N_{C_2} - 2N_{C_2'} - 2N_{C_2''},
  \]
  where~$C_2,C_2',C_2''$ are the three subgroups of order~$2$ of~$G$.
  Simplified, it becomes
  \[
    3 = N_{C_2} + N_{C_2'} + N_{C_2''} + N_{C_3} - N_G,
  \]
  a special case of Example~\ref{ex:rel-pq}.
\end{example}

When $\Hsg$ is the set of all nontrivial subgroups of
$G$, we have the following simple characterization for the existence
of Brauer relations.

\begin{theorem}[Funakura]\label{thm:funakura}
  The group $G$ admits a useful Brauer relation if and only if~$G$ contains a non-cyclic
  subgroup of order~$pq$, where~$p$ and~$q$ are primes (not necessarily
  distinct).
\end{theorem}

\begin{proof}
  This is \cite[Theorem 9]{Funakura1978}.
\end{proof}

We now turn to the question of existence for norm relations.
Together with Theorem~\ref{thm:funakura} this will at the same time show that there are in general more norm relations than scalar norm relations.
As a first step towards the classification, we formulate a representation theoretic criterion.
In the following we denote by $e_1,\dotsc,e_r$ the central primitive idempotents of the group algebra $\QQ[G]$.

\begin{proposition}\label{prop:generel}
  Let~$\Hsg$ be a set of nontrivial subgroups of a finite group~$G$.
  Let~$\QQbar$ be the algebraic closure of~$\Q$ in~$\C$.
  Then the following are equivalent:
  \begin{enumerate}
    \item\label{item:generel} there exists a norm relation in~$G$ with respect
      to $\Hsg$;
    \item\label{item:tsideal} we have $\langle N_H \mid H\in\Hsg\rangle_{\QQ[G]} = \QQ[G]$ (as a two
      sided ideal);
    \item\label{item:idemp} for all~$i=1,\dots,r$, there exists~$H\in\Hsg$ such
      that~$e_iN_H\neq 0$;
    \item\label{item:repsQ} for every simple $\QQ[G]$-module $V$, there exists~$H\in\Hsg$ such that the space of fixed
      points~$V^H$ is nonzero;
    \item\label{item:repsQbar} for every simple $\QQbar[G]$-module $V$, there exists~$H\in\Hsg$ such that the space of fixed
      points~$V^H$ is nonzero;
    \item\label{item:repsC} for every simple $\CC[G]$-module $V$, there exists~$H\in\Hsg$ such that the space of fixed
      points~$V^H$ is nonzero;
    \item\label{item:sphere} for every unitary~$\CC[G]$-module~$V$, there
      exists~$H\in\Hsg$ such that~$H$ has a fixed point on the
      unit sphere of~$V$ with respect to the invariant Hermitian norm.
  \end{enumerate}
\end{proposition}
\begin{proof}
  The set of elements of the form~$\sum_{i=1}^\ell a_i N_{H_i} b_i$
  with~$a_i,b_i\in\QQ[G]$ and~$H_i\in\Hsg$ is exactly the two-sided
  ideal~$\langle N_H \mid H\in\Hsg\rangle_{\QQ[G]}$. Moreover a two-sided
  ideal contains~$1$ if and only if it equals the whole ring. This proves the
  equivalence between~(\ref{item:generel}) and~(\ref{item:tsideal}).

  For every two-sided ideal~$J$ of~$\QQ[G]$ we have~$J = \sum_{i=1}^r e_i J$,
  so~$J = \QQ[G]$ if and only if~$e_iJ = e_i\QQ[G]$ for every~$i \in \{1,\dotsc,r\}$. In addition,
  $e_iJ$ projects isomorphically to a two-sided ideal in the simple
  algebra~$\QQ[G]/(1-e_i)$, so it is either
  equal to~$e_i\QQ[G]$ or zero. Applying this to~$J = \langle N_H \mid
  H\in\Hsg\rangle_{\QQ[G]}$, and noting that~$e_i J = 0$ if and only
  if~$e_i N_H = 0$ for all~$H\in\Hsg$, this proves the equivalence
  between~(\ref{item:tsideal}) and~(\ref{item:idemp}).

  For every~$\QQ[G]$-module~$V$ and subgroup~$H \leq G$, we
  have~$(\frac{1}{|H|}N_H) \cdot V = V^H$.
  Let~$1 \leq i\le r$, and let~$V_i$ be the (up to isomorphism) unique simple~$\Q[G]$-module
  such that~$e_iV_i\neq 0$. %
  Since the simple algebra~$\Q[G]/(1-e_i)$ acts faithfully on~$V_i$,
  we have
  \[
    e_iN_H=0
    \Longleftrightarrow N_H \cdot V_i=0
    \Longleftrightarrow \Bigl(\frac{1}{|H|}N_H\Bigr) \cdot V_i=0
    \Longleftrightarrow V_i^H=0.
  \]
  This proves the
  equivalence between~(\ref{item:idemp}) and~(\ref{item:repsQ}).

  Let~$K\subseteq L$ be subfields of~$\C$. For every~$K[G]$-module~$V$ and
  every subgroup~$H\le G$, we have~$\dim_K V^H = \dim_L (V\otimes_K L)^H$; in
  particular~$V^H\neq 0$ if and only if~$(V\otimes_K L)^H\neq 0$. In addition,
  every simple~$L[G]$-module is isomorphic to a submodule of~$V\otimes_K L$ for
  some simple~$K[G]$-module~$V$. Applying this to~$\QQ\subseteq\QQbar$
  and~$\QQbar\subseteq\C$, we obtain~(\ref{item:repsC})
  $\Rightarrow$~(\ref{item:repsQbar}) $\Rightarrow$~(\ref{item:repsQ}).

  Let~$W$ be a simple~$\QQbar[G]$-module, and let~$V$ be a simple~$\Q[G]$-module
  such that~$W$ is isomorphic to a submodule of~$V\otimes_\QQ \QQbar$.
  Let~$V\otimes_\QQ \QQbar \cong
  \bigoplus_{j=1}^k W_j$ be a decomposition into simple~$\QQbar[G]$-modules,
  so that~$W$ is isomorphic to one of the~$W_j$.
  Since the~$W_j$ are pairwise Galois conjugate, we have~$\dim_{\QQbar} W_j^H =
  \dim_{\QQbar} W_1^H$ for all~$j$, so that~$V^H\neq 0$ implies that for all~$j$
  we have~$W_j^H\neq 0$. In particular~$W^H\neq 0$ and we get~(\ref{item:repsQ})
  $\Rightarrow$~(\ref{item:repsQbar}).

  The simple~$\C[G]$-modules are exactly the~$V \otimes_{\QQbar} \C$ where~$V$
  ranges over the simple~$\QQbar[G]$-modules, so that we
  have~(\ref{item:repsQbar}) $\Rightarrow$~(\ref{item:repsC}).

  Let us prove that~(\ref{item:repsC}) implies~(\ref{item:sphere}).
  Let~$V$ be a unitary $\CC[G]$-module. It contains a
  simple~$\CC[G]$-submodule~$V'$, and therefore by~(\ref{item:repsC}) there
  exists~$H\in\Hsg$ and~$v\in (V')^H\setminus\{0\}$, so that~$v\in
  V^{H}$. Then~$v/\|v\|$ is a fixed point of~$H$ on the unit sphere of~$V$.
  Conversely, since every~$\C[G]$-module is unitarizable, (\ref{item:sphere}) implies~(\ref{item:repsC}).
\end{proof}

This can in turn be used to characterize groups that admit norm relations.

\begin{theorem}\label{thm:genrel}
  Let $G$ be a finite group. Then the following are equivalent:
  \begin{enumerate}
    \item\label{item:generelall} the group $G$ admits a norm relation;
    \item\label{item:generelcyc} the group $G$ admits a norm
      relation with respect to the set of nontrivial
      cyclic subgroups of~$G$;
    \item\label{item:wolf} the group $G$ has a non-cyclic subgroup of
      order~$pq$, where~$p$ and~$q$ are prime, or a subgroup isomorphic
      to~$\SL_2(\F_p)$ where~$p>5$ is prime;
    \item\label{item:wolfsmall} the group $G$ contains a non-cyclic subgroup of order~$pq$, where~$p$ and~$q$ are
      prime, or a subgroup isomorphic to~$\SL_2(\F_p)$ where~$p=2^{2^k}+1$ is a Fermat
      prime with~$k>1$.
  \end{enumerate}
\end{theorem}

\begin{proof}
  Clearly~(\ref{item:generelcyc}) implies~(\ref{item:generelall}).
  The converse follows from Remark~\ref{rem:defrels}~(\ref{item:replace}).

  Applying criterion~(\ref{item:sphere}) of Proposition~\ref{prop:generel}, we see
  that~(\ref{item:generelcyc}) is equivalent to
  the nonexistence of a unitary~$\CC[G]$-module~$V$ such that for every~$g\neq
  1$, the element~$g$ does not have fixed points on the unit sphere of~$V$,
  in other words such that~$G$ acts freely on the unit sphere of~$V$.
  The equivalence between this last statement and~(\ref{item:wolf}) is Wolf's
  theorem~(\cite[Theorem 6.1]{Wall2013}).

  The equivalence between~(\ref{item:wolf}) and~(\ref{item:wolfsmall}) follows from
  observing that when~$p$ is not a Fermat prime, we may pick a prime~$q\neq 2$
  dividing~$p-1$ and an element~$a\in \F_p^\times$ of order~$q$, and that the
  subgroup of~$\SL_2(\F_p)$ generated
  by~$(\begin{smallmatrix}a & 0 \\ 0 & a^{-1}\end{smallmatrix})$
  and~$(\begin{smallmatrix}1 & 1 \\ 0 & 1\end{smallmatrix})$
  is a noncyclic group of order~$pq$.
\end{proof}

\begin{example}
  In view of Theorem~\ref{thm:genrel} compared with Theorem~\ref{thm:funakura},
  the smallest group that admits a norm relation but no scalar norm relation is
  the group~$\SL_2(\F_{17})$ of cardinality~$4896$. It admits a norm relation with
  denominator~$17$ with respect to the set of subgroups of index at most~$1632$.
\end{example}

Even if a group admits both a scalar norm relation and a norm relation,
there might still be a difference when it comes to the subgroups that are involved in the relations.
The following example illustrates this phenomenon.

\begin{example}\label{ex:normdiff1}
  Consider the direct product~$G = C_2 \times \mathrm{SU}_3(\FF_2)$ of order~$432$.
  Then the smallest~$n\ge 1$ such that~$G$ admits a scalar norm relation with
  respect to the set of subgroups of index at most~$n$ is~$72$, but the
  smallest~$n$ such that~$G$ admits a norm relation with respect to
  the set of subgroups of index at most~$n$ is~$54$.
\end{example}

\begin{remark}
  In addition to an existence criterion like Theorem~\ref{thm:genrel} or
  Proposition~\ref{prop:generel}, it would be interesting to establish a complete
  classification of norm relations similar to the existing one for Brauer
  relations~\cite{Bartel2015, Bartel2014}.
\end{remark}

\subsection{Arithmetic properties of relations}

\begin{definition}\label{def:optiden}
  Let~$\Hsg$ be a set of nontrivial subgroups of~$G$.
  We define the \emph{optimal denominator~$d(\Hsg)$ relative to
  $\Hsg$}, to be the unique nonnegative integer such that
  \[
    d(\Hsg)\ZZ = \ZZ \cap \langle N_H \mid
    H\in\Hsg\rangle_{\ZZ[G]}.
  \]
  Let
  \[
    1 = \sum_{i=1}^\ell a_i N_{H_i} b_i
  \]
  be a norm relation with~$H_i\in \Hsg$ and $a_i, b_i \in \Q[G]$.
  The least common denominator of the~coefficients of the~$a_i$ and~$b_i$ is called the \emph{denominator} of the
  relation.
\end{definition}

\begin{remark}
  We have~$d(\Hsg)>0$ if and only if there exists a norm relation over~$\Q$.
  In that case, the optimal denominator divides the denominator of every
  relation, and there exists a relation with optimal denominator.
\end{remark}

For arithmetic applications, it is desirable to have a relation with
denominator as small as possible, and more precisely with denominator divisible
by as few primes as possible (see Corollary~\ref{cor:relmaxorder},
Corollary~\ref{cor:relSunits} and Proposition~\ref{prop:apl_normrel}). The
following proposition characterizes the existence of relations with denominator
coprime to a given~$p$. In addition, Theorem~\ref{thm:denrelG} says that the primes
that do not divide~$|G|$ can always be removed from the denominator of
norm relations.

\begin{remark}
  Consider a scalar norm relation $\rela$ of the form $0 = \sum_{H \in \Hsg}
  a_H N_{H}$ with $a_H \in \Z$. Since $1 = \sum_{H\in\Hsg} -\frac{a_H}{a_1} N_H$, we will view $\rela$
  as a norm relation and define its denominator to be the denominator of the
  corresponding norm relation.
  Thus any scalar norm relation with denominator $d$ is of the form
  \[ d = \sum_{H \in \Hsg} b_H N_H \]
  with $d, b_H \in \Z$ coprime.
\end{remark}

\begin{proposition}\label{prop:dencoprime}
  Let~$\Hsg$ be a set of nontrivial subgroups of~$G$, and let~$p$ be a
  prime number. Let~$J$ be the Jacobson radical of~$\FF_p[G]$.
  Then the following are equivalent:
  \begin{enumerate}
    \item\label{item:denom-nop} $p\nmid d(\Hsg)$;
    \item\label{item:modprel} there exists a norm relation over~$\FF_p$ with
      respect to~$\Hsg$;
    \item\label{item:modpJrel} there exists an identity of the form
      \[
        1 = \sum_i a_i N_{H_i} b_i
      \]
      where~$a_i,b_i\in \FF_p[G]/J$ and the identity holds in~$\FF_p[G]/J$;
    \item\label{item:modprelreps} for every simple~$\F_p[G]$-module~$V$, there
      exists~$H\in\Hsg$ such that~$N_H\cdot V\neq 0$;
    \item\label{item:Fpbarrelreps} for every
      simple~$\overline{\F}_p[G]$-module~$V$, there exists~$H\in\Hsg$
      such that~$N_H\cdot V\neq 0$.
  \end{enumerate}
\end{proposition}

\begin{proof}
  It is clear that~(\ref{item:denom-nop}) implies~(\ref{item:modprel}).
  Conversely, assume that
  \[
    1 = \sum_i \bar{a}_i N_{H_i} \bar{b}_i
  \]
  is a relation over~$\FF_p$. Pick arbitrary lifts~$a_i,b_i\in\Z[G]$
  of~$\bar{a}_i,\bar{b}_i$, and let
  \[
    \delta = \sum_i a_i N_{H_i} b_i.
  \]
  We have~$N_{\Z[G]/\Z}(\delta) \equiv N_{\FF_p[G]/\FF_p}(1) \equiv 1 \bmod p$,
  which is nonzero. Therefore the norm is nonzero, the element~$\delta$ is
  invertible in~$\Q[G]$ and the denominator~$d$ of~$\delta^{-1}$ is coprime to~$p$.
  We therefore obtain the relation
  \[
    d = \sum_i (d\delta^{-1})a_i N_{H_i} b_i
  \]
  with~$d\in\Z$ coprime to~$p$ and $(d\delta^{-1})a_i\in\Z[G]$, and
  therefore~$p\nmid d(\Hsg)$.
  This proves that~(\ref{item:modprel}) implies~(\ref{item:denom-nop}).

  It is clear that~(\ref{item:modprel}) implies~(\ref{item:modpJrel}).
  Conversely, assume that
  \[
    1 = \sum_i \bar{a}_i N_{H_i} \bar{b}_i
  \]
  holds in~$\FF_p[G]/J$.
  Pick arbitrary lifts~$a_i,b_i\in\FF_p[G]$
  of~$\bar{a}_i,\bar{b}_i$, and let
  \[
    \delta = \sum_i a_i N_{H_i} b_i.
  \]
  We have~$\delta \equiv 1 \bmod J$; since~$1$ is invertible and~$J$ is a nilpotent
  two-sided ideal, this implies that~$\delta$ is invertible.
  We therefore have the relation
  \[
    1 = \sum_i \delta^{-1}a_i N_{H_i} b_i
  \]
  in~$\FF_p[G]$.
  This proves that~(\ref{item:modpJrel}) implies~(\ref{item:modprel}).

  The proof of the equivalence between~(\ref{item:modpJrel})
  and~(\ref{item:modprelreps}) is identical to that of
  Proposition~\ref{prop:generel} by considering the central primitive
  idempotents of the semisimple algebra~$\FF_p[G]/J$.

  The proof of the equivalence between~(\ref{item:modprelreps})
  and~(\ref{item:Fpbarrelreps}) is identical to that of
  Proposition~\ref{prop:generel}.
\end{proof}

\begin{remark}
  It would be interesting to find a general existence criterion similar to
  Theorem~\ref{thm:genrel} for norm relations over~$\FF_p$.
\end{remark}

\begin{theorem}\label{thm:denrelG}
  Let~$\Hsg$ be a set of nontrivial subgroups of~$G$.
  If~$d(\Hsg)>0$ then~$d(\Hsg)$ divides~$|G|^3$.
\end{theorem}

\begin{proof}
  The following proof we will use properties of maximal orders in semisimple algebras,
  which can be found in~\cite{Reiner}.
  Assume that~$d(\Hsg)>0$, and let~$p$ be a prime number.
  In the following we will denote by~$\Z_p$ and~$\Q_p$ the ring of $p$-adic
  integers and the field of $p$-adic numbers respectively.
  Let~$\order$ be a maximal order of~$\Q_p[G]$
  containing~$\Z_p[G]$, and
  let~$e_1,\dots,e_r$ be central primitive idempotents of~$\Q_p[G]$ contained
  in~$\order$, which exist since~$\order$ is a maximal
  order~\cite[Theorem~(10.5)~(i)]{Reiner}.

  Let~$1\le i\le r$. By Proposition~\ref{prop:generel}, there exists~$H = H_i\in\Hsg$
  such that~$e_iN_H\neq 0$. Let~$N_i = e_iN_H$, which satisfies~$N_i^2 =
  |H|\cdot N_i$.
  By~\cite[Theorem~(17.3)~(ii)]{Reiner} there is an isomorphism~$\psi\colon
  \order/(1-e_i) \to \mat_{n}(\Lambda)$ where~$\Lambda = \Lambda_i$ is the maximal
  order of a division algebra over~$\Q_p$ and~$n = n_i\ge 1$; let~$v$ be the
  normalized valuation of~$\Lambda$.
  We extend~$\psi$ to~$\order$ via~$\order \to \order/(1-e_i)$.
  Write the Smith normal form (see~\cite[Theorem~(17.7)]{Reiner}) of~$\psi(N_i)$
  as follows:
  let~$U,V\in\mat_n(\Lambda)^\times$
  and~$\lambda_1,\dots,\lambda_n\in \Lambda$ be such that~$U\psi(N_i)V$ is the
  diagonal matrix~$(\lambda_1,\dots,\lambda_n)$ and~$v(\lambda_j)\le
  v(\lambda_{j+1})$ for all~$1\le j<n$.
  Since~$\lambda_1\neq 0$, the relation~$N_i^2 = |H|\cdot N_i$ implies
  that~$v(\lambda_1)\le v(|H|)$.
  Let~$u_i,v_i\in\order$ be such that~$\psi(u_i)=U$ and~$\psi(v_i)=V$.
  For each~$1\le j\le n$, let~$a_{i,j},b_{i,j}\in\order$ be such
  that~$\psi(a_{i,j})\in \mat_n(\Lambda)$ is the matrix with all
  coefficients~$0$ except the~$(j,1)$-th coefficient equal to~$1$
  and~$\psi(b_{i,j})\in \mat_n(\Lambda)$ is the matrix with all
  coefficients~$0$ except the~$(1,j)$-th coefficient equal
  to~$|H|\lambda_1^{-1}\in \Lambda$.
  We obtain that~$\sum_{j=1}^n \psi(a_{i,j}u_i N_i v_i b_{i,j})$
  is~$|H|$ times the~$n\times n$ identity matrix, and therefore
  \[
    \frac{1}{|H_i|}\sum_{j=1}^n e_ia_{i,j}u_i N_i v_i b_{i,j} = e_i.
  \]
  Summing over~$i$, we obtain the norm relation
  \[
    \sum_{i=1}^r\sum_{j=1}^{n_i}e_ia_{i,j}u_i\frac{1}{|H_i|}N_{H_i}v_ib_{i,j} = 1.
  \]
  Since~$e_ia_{i,j}u_i\in\order$, $v_ib_{i,j}\in\order$,
  and~$\order\subseteq\frac{1}{|G|}\Z_p[G]$ (\cite[(27.1) Proposition]{Curtis1990}),
  the denominator of this relation
  divides~$|G|^3$ in~$\Z_p$, so that~$|G|^3\in d(\Hsg)\Z_p$.

  Putting all~$p$ together, we obtain that~$d(\Hsg)$ divides~$|G|^3$ as claimed.
\end{proof}

\begin{remark}
  It is clear from the proof that~$|G|^3$ can be replaced with~$hg^2$, where~$h$
  is the least common multiple of the~$|H|$ for~$H\in\Hsg$ and~$g>0$ is the
  smallest integer such that there is a maximal order~$\order$
  satisfying~$\Z[G]\subset \order \subset \frac{1}{g}\Z[G]$.
\end{remark}

The following example shows that in general the minimal denominators of scalar
and arbitrary norm relations are not equal.

\begin{example}
  Let~$G = A_5$ be the alternating group on $5$ letters, and let~$\Hsg$ be the set of subgroups of $G$ of index at
  most~$12$ (up to conjugacy, these subgroups are $C_5, S_3, D_5, A_4, A_5$). Then~$G$ admits a
  scalar norm relation with respect to~$\Hsg$. However, all scalar norm
  relations with respect to~$\Hsg$ have denominator supported at~$2,3$
  and~$5$, but~$G$ admits a norm relation with respect
  to~$\Hsg$ with denominator supported only at~$2$ and~$5$.
\end{example}

\subsection{Norm relations in finite abelian groups}

In the case of abelian groups, there is a second way to turn Brauer relations
into norm relations and conversely based on duality.
\begin{definition}
  Let~$G$ be a finite abelian group. Let~$\widehat{G} = \Hom(G,\C^\times)$ be
  the dual of~$G$. We have a canonical isomorphism~$G \to
  \widehat{\widehat{G}}$ given by~$g \mapsto (\chi \mapsto \chi(g))$,
  and a noncanonical isomorphism~$G\cong\widehat{G}$.
  Let~$X\subseteq G$ be a subset; we write~$X^\perp = \{\chi\in \widehat{G} \mid
  \chi(x)=1 \text{ for all } x\in X\}\subseteq \widehat{G}$.
\end{definition}

In the following, whenever we are dealing with an abelian group $G$, we will
use the canonical isomorphism with its bidual and its inverse implicitly to
identify subgroups of the dual of $\widehat G$ with subgroups of $G$.
Since the $1$-dimensional characters of~$G$ form a $\C$-basis of the space
of class functions of~$G$, this space is canonically isomorphic
to~$\C[\widehat{G}]$; we will also use this identification implicitly.

\begin{proposition}\label{prop:abeliandual}
  Let~$G$ be a finite abelian group.
  \begin{enumerate}
    \item Let~$H\le G$ be a subgroup. We have~$\Ind_{G/H}(\triv_H) = N_{H^\perp}$.
    \item We have
      \[
        \sum_{H\le G}a_H\cdot \Ind_{G/H}(\triv_H)=0 \quad \text{(Brauer relation of }G)
      \]
      if and only if we have
      \[
        \sum_{H\le \widehat{G}}a_{H^\perp}\cdot N_{H}=0 \quad \text{(norm relation of }\widehat{G}).
      \]
      The second equality is a norm relation if and only if~$a_G\neq 0$.
    \item Let~$\Hsg$ be a set of subgroups of~$G$, and
      let~$\Hsg^\perp = \{H^\perp \colon H\in \Hsg\}$.
      Then the group~$G$ admits a Brauer relation with respect to~$\Hsg$ if
      and only if~$\widehat{G}$ admits a norm relation with respect
      to~$\Hsg^\perp$.
  \end{enumerate}
\end{proposition}
\begin{proof}
  Let~$\chi\in\widehat{G}$. By Frobenius reciprocity we have
  \[
    \langle \Ind_{G/H}(\triv_H), \chi\rangle_G = \langle \triv_H,
    \Res_{G/H}(\chi)\rangle_H.
  \]
  This inner product equals~$1$ if~$\Res_{G/H}(\chi)=\triv_H$, i.e. if~$\chi\in
  H^\perp$, and~$0$ otherwise, proving the first assertion.
  The next two follow trivially.
\end{proof}

\begin{remark}
  Obviously, we have the corresponding dual statements as follows.
  \begin{enumerate}
    \item Let~$H\le G$ be a subgroup.
      We have~$N_{H} = \Ind_{\widehat{G}/H^\perp}(\triv_{H^\perp})$.
    \item We have
      \[
        \sum_{H\le G}a_H\cdot N_{H}=0 \quad \text{ (norm relation of }G)
      \]
      if and only if we have
      \[
        \sum_{H\le \widehat{G}}a_{H^\perp}\cdot \Ind_{\widehat{G}/H}(\triv_{H})=0 \quad
        \text{(Brauer relation of }\widehat{G}).
      \]
      The Brauer relation is useful if and only if~$a_G\neq 0$.
    \item The group~$G$ admits a norm relation with respect to~$\Hsg$ if
      and only if~$\widehat{G}$ admits a Brauer relation with respect
      to~$\Hsg^\perp$.
  \end{enumerate}
\end{remark}

\begin{proposition}\label{prop:funakurarel}
  Let~$\mu$ denote the M\"obius function. For~$n>1$ an integer, let~$\rad(n) =
  \prod_{p\mid n}p$, where the product ranges over prime divisors~$p$ of~$n$.

  Let~$G$ be a non-cyclic abelian group.
  \begin{enumerate}
  \item We have the norm relation~$\rela_G$:
    \[
      1 = \sum_{C = \langle\chi\rangle\le \widehat{G} \text{ cyclic}} a_{\ker\chi} N_{\ker \chi},
    \]
    where
    \[
      a_{\ker\chi} = \frac{1}{|\ker \chi|}\sum_{C\le C'\le \widehat{G}\text{
        cyclic}}\mu([C' : C]).
    \]
  \item\label{item:altakerchi} We have
    \[
      a_{\ker\chi}
      = \frac{c}{|G|}\prod_{p\mid c} \left(1-p^{r_p-1}\delta_{\chi,p}\right)
      \prod_{p\mid|G|,\, p\nmid c} \left(-p-p^2-\dots-p^{r_p-1}\right)
    \]
      where~$c$ denotes the order of~$\chi$ and in each product~$p$ ranges over prime divisors,
      where~$\delta_{\chi,p}=1$ or~$0$ according as
      whether there exists~$\chi'\in\widehat{G}$ such that~$(\chi')^p=\chi$, and
      where~$r_p = \dim_{\F_p}(G/G^p)$ denotes the~$p$-rank of~$G$.

  \item\label{item:funakuraden} The denominator of~$\rela_G$ divides
    $
      \frac{|G|}{\rad(|G|)},
    $
    with equality if~$G$ is a $p$-group.
  \end{enumerate}
\end{proposition}
\begin{proof}
  The first assertion is~\cite[Corollary 6]{Funakura1978}, applied to the group~$\widehat{G}$ and
  dualized using Proposition~\ref{prop:abeliandual}.

  In order to prove~(\ref{item:altakerchi}), we rewrite the expression for~$a_{\ker\chi}$ as follows.
  \[
    \frac{1}{|\ker \chi|}\sum_{C\le C'\le \widehat{G}\text{
      cyclic}}\mu([C' : C])
    = \frac{c}{|G|}\sum_{d\ge 1}\mu(d)|\{C\le C'\le\widehat{G} \mid [C':C]=d\}|.
  \]
  Every~$C'$ that appears in this sum is generated by an element~$\chi'$ of
  order~$cd$ such that~$(\chi')^d=\chi$. Moreover, the set of~$\chi''\in\widehat{G}$
  that generate the same cyclic group as~$\chi'$ and
  satisfy~$(\chi'')^d=\chi$ is exactly the set of~$\chi'' = (\chi')^k$
  where~$k\in(\Z/cd\Z)^\times$ is such that~$k\equiv 1\bmod{c}$: there are
  exactly~$\varphi(cd)/\varphi(c)$ such elements.
  We therefore obtain
  \[
    a_{\ker\chi}
    = \frac{c}{|G|}\sum_{d\mid |G|}
    \mu(d)\frac{\varphi(c)}{\varphi(cd)}|\{\chi'\in\widehat{G} \mid (\chi')^d = \chi
    \text{ and }\chi'\text{ has order }cd\}|.
  \]
  This expression is multiplicative with respect to the decomposition of~$G$
  into a product of~$p$-Sylow subgroups, so we may assume that~$G$ is a
  nontrivial $p$-group, in
  which case~$p\mid c$ if and only if~$\chi\neq 1$. Each sum then has exactly
  two nonzero summands corresponding
  to~$d=1$ and~$d=p$, and the~$d=1$ term in the sum is~$1$, so it suffices to
  evaluate the summand~$d=p$.
  If~$\chi\neq 1$, then every~$\chi'$ such that~$(\chi')^p=\chi$ has order~$pc$,
  and the number of such elements is~$|\widehat{G}[p]|\delta_{\chi,p} =
  p^{r_p}\delta_{\chi,p}$; moreover~$\varphi(c)/\varphi(pc) = 1/p$.
  If~$\chi=1$, then~$c=1$ and every~$\chi'$ that has order~$p$ satisfies~$(\chi')^p=\chi$,
  and the number of such elements is~$|\widehat{G}[p]|-1 = p^{r_p}-1$;
  moreover~$\varphi(c)/\varphi(pc) = 1/(p-1)$; finally we have
  \[
    1 - \frac{p^{r_p}-1}{p-1} = -p -p^2 - \dots - p^{r_p-1},
  \]
  completing the proof of~(\ref{item:altakerchi}).

  Let~$p$ be a prime divisor of~$|G|$ and~$\chi\in\widehat{G}$. By inspection, we see that the valuation
  of~$a_{\ker\chi}$ satisfies~$v_p(a_{\ker\chi})\ge 1-v_p(|G|)$ if~$p\nmid c$,
  and~$v_p(a_{\ker\chi}) \ge v_p(c)-v_p(|G|)\ge 1-v_p(|G|)$ if~$p\mid c$.
  In particular, we always have~$v_p(a_{\ker\chi})\ge 1-v_p(|G|)$, and if~$G$ is
  a $p$-group then there is equality for~$\chi=1$; this proves the claim about
  the denominator of~$\rela_G$.
\end{proof}

\begin{example}
  Let~$G = C_{18}\times C_2$. Then the denominator of~$\rela_G$ is~$2$,
  but we have~$\tfrac{|G|}{\rad(|G|)} = 6$. This shows that equality in~(\ref{item:funakuraden})
  does not always hold.
\end{example}

We can leverage the previous proposition to obtain optimal relations with
respect to the denominator and the index of the subgroups involved in the case
of abelian groups.

\begin{theorem}\label{thm:optimalabelianrels}
  Let~$G$ be a finite abelian group, and write~$G\cong C\times Q$ where~$C$ is
  the largest cyclic factor of~$G$.
  \begin{enumerate}
    \item\label{item:denom1} Denominator~$1$ case.
      \begin{enumerate}
        \item\label{item:denom1cond} The group~$G$ admits a denominator~$1$ norm relation if and only if~$|Q|$
          is divisible by at least two distinct primes.
        \item\label{item:denom1index}
          Assume that~$G$ admits a denominator~$1$ norm relation.
          The smallest~$n\ge 1$ such that~$G$ admits a denominator~$1$
          norm relation with respect to the set of subgroups of index at most~$n$ is
          \[
            n_0 = |C|\cdot \max_p |Q_p|,
          \]
          where~$p$ ranges over all prime numbers.
        \item\label{item:denom1rel}
          Assume that~$G$ admits a denominator~$1$ norm relation.
          Let~$\Hsg$ be the union over the prime divisors~$p$ of~$|G|$
          of the set of subgroups~$H$
          of $G_{p'}$ such that~$G_{p'}/H$ is cyclic. Every subgroup
          in~$\Hsg$ has index at most~$n_0$.
          For each prime number~$p$ dividing~$|G|$, let~$d_p$ be the
          denominator of~$\rela_{G_{p'}}$, and let~$1 = \sum_p u_pd_p$ be a
          B\'ezout identity for the~$d_p$. Then
          \[
            \sum_p u_p d_p\rela_{G_{p'}}
          \]
        is a denominator~$1$ scalar norm relation with respect to~$\Hsg$.
      \end{enumerate}
    \item\label{item:denompp} Prime power denominator case. Assume that~$Q$ is a~$p$-group.
      \begin{enumerate}
        \item\label{item:denomppcond} The group~$G$ admits a norm relation if and only if~$Q\neq 1$.
        \item\label{item:denomppindex}
          Assume that~$G$ admits a norm relation.
          The smallest~$n\ge 1$ such that~$G$
          admits a norm relation with respect to
          the set of subgroups of index at most~$n$ is~$n_0 = |C|$.
        \item\label{item:denompprel}
          Assume that~$G$ admits a norm relation.
          Let~$\Hsg$ be the set of subgroups~$H$ of~$G$ such that~$G/H$
          is cyclic. Every subgroup in~$\Hsg$ has index at most~$n_0$.
          Then~$\rela_{G_p}$ is a scalar norm relation with respect
          to~$\Hsg$ and with denominator~$|G_p|/p$.
      \end{enumerate}
  \end{enumerate}
\end{theorem}
\begin{proof}
  Let~$\Hsg$ be a set of subgroups of~$G$.
  Let~$q$ be a prime, and assume that~$G$ admits a norm relation
  with respect to~$\Hsg$ and
  with denominator not divisible by~$q$.
  We claim that~$\Hsg$ contains a subgroup of~$G$ of index at
  least~$|C|\cdot |Q_q|$.
  To prove the claim, choose an isomorphism~$G\cong G_q\times G_{q'}$,
  let~$\chi\colon G \to \overline{\F}_q^\times$ be a
  one-dimensional character of maximal order, and let~$V$ be the
  corresponding~$\overline{\F}_q[G]$-module. Clearly~$G_q\subseteq \ker\chi$
  and~$\chi$ has order~$|C_{q'}|$.
  By Proposition~\ref{prop:dencoprime} (\ref{item:Fpbarrelreps}), there
  exists~$H\in\Hsg$ such that~$N_H\cdot V \neq 0$.
  Since~$H_q = H\cap G_q$ acts trivially on~$V$ and~$q\cdot V=0$, by
  Remark~\ref{rem:defrels}~(\ref{item:replace}) we have~$H_q=1$, and in
  particular the index of~$H$ in~$G$ is divisible by~$|G_q|$.
  Since~$|H|$ is not divisible by~$q$, we have
  \[
    N_H\cdot V \neq 0
    \Longleftrightarrow \frac{1}{|H|}N_H\cdot V \neq 0
    \Longleftrightarrow V^H \neq 0
    \Longleftrightarrow H \subseteq \ker\chi.
  \]
  In particular, the index of~$H$ in~$G$ is divisible by the order~$|C_{q'}|$ of~$\chi$.
  Since~$|G_q|$ and~$|C_{q'}|$ are coprime, the index of~$H$ is therefore
  divisible by~$|C_{q'}|\cdot |G_q| = |C| \cdot |Q_q|$ as claimed.

  Applying the claim to~$q=p$ for each prime divisor~$p$ of~$|G|$ proves that
  in~(\ref{item:denom1index}) the integer~$n_0$ is indeed a lower bound,
  and that in~(\ref{item:denom1cond}) the ``only if'' direction holds.

  Let~$p$ be a prime number dividing~$|G|$ and let~$H\le G_{p'}$ be a subgroup such that~$G_{p'}/H$ is
  cyclic. Then~$|G_{p'}/H|$ divides~$|C_{p'}|$, so the index of~$H$ in~$G$
  divides~$|C_{p'}|\cdot|G_p| = |C|\cdot |Q_p| \le n_0$ as claimed
  in~(\ref{item:denom1rel}). The existence of the B\'ezout identity follows from
  the fact that by Proposition~\ref{prop:funakurarel} (\ref{item:funakuraden}),
  for each prime number~$p$ dividing~$|G|$ the denominator~$d_p$ is not divisible by~$p$, and
  all~$d_p$ are divisors of~$|G|$. This proves~(\ref{item:denom1rel}), and
  therefore completes~(\ref{item:denom1cond}) and~(\ref{item:denom1index}).

  Now assume that~$Q$ is a~$p$-group.
  If~$G$ admits a norm relation, then applying the above claim to a
  prime~$q$ that does not divide the denominator of the norm relation or~$|G|$ proves
  that in~(\ref{item:denomppindex}) the integer~$n_0$ is indeed a lower bound,
  and that in~(\ref{item:denomppcond}) the ``only if'' direction holds.

  Let~$H\le G$ be a subgroup such that~$G/H$ is cyclic. Then~$|G/H|$
  divides~$|C| = n_0$ as claimed in~(\ref{item:denompprel}).
  The rest of~(\ref{item:denompprel}) is contained in
  Proposition~\ref{prop:funakurarel}, and therefore
  completes~(\ref{item:denomppcond}) and~(\ref{item:denomppindex}).
\end{proof}

\section{Arithmetic applications}\label{sec:applications}

Let $K/F$ be a normal extension of algebraic number fields with Galois group $G$.
In this section we will discuss the consequences of the existence of
norm relations of~$G$, scalar or not,
for the structure and arithmetic properties of $K$.

In this section, we will consider either a scalar norm relation of the form
\begin{align}\label{eq:normrel}
  d = \sum_{i=1}^\ell a_i N_{H_i}\tag{$\star$}
\end{align}
with $H_i \leq G$, $d\in\Z_{>0}$ and~$a_i \in \Z$, or a norm relation
\begin{align}\label{eq:normrelgen}
  d = \sum_{i=1}^\ell a_i N_{H_i} b_i\tag{$\star\star$}
\end{align}
with $H_i \leq G$, $d \in \Z_{>0}$, $a_i,b_i \in \Z[G]$.

We begin by describing a general statement that holds for arbitrary $\Z[G]$-modules.
Let~$M$ be a $\Z[G]$-module and~$H \leq G$ a subgroup; the action of the norm
$N_H$ induces a map $M \to M^H$, which we also denote by $N_H$.

\begin{proposition}\label{prop:meta}
  Let~$M$ be a $\Z[G]$-module.
  \begin{enumerate}
    \item
      If $G$ admits a scalar norm relation of the form~(\ref{eq:normrel}), then
      the exponent of the quotient $M/\sum_{i=1}^\ell M^{H_i}$ is finite and divides $d$.
    \item
      If $G$ admits a norm relation of the form~(\ref{eq:normrelgen}), then
      the exponent of the quotient $M/\sum_{i=1}^\ell a_i M^{H_i}$ is finite and divides $d$.
  \end{enumerate}
\end{proposition}

\begin{proof}
  Let $m \in M$.
  In the first case we have
  \[
    d\cdot m = \left( \sum_{i=1}^\ell a_i N_{H_i}\right)\cdot m =
    \sum_{i=1}^\ell \bigl(a_i N_{H_i}\cdot m\bigr) \in \sum_{i=1}^\ell M^{H_i},
  \]
  whereas in the second case we have (using the $G$-invariance of $M$)
  \[
    d\cdot m = \left( \sum_{i=1}^\ell a_i N_{H_i}b_i\right)\cdot m
    = \sum_{i=1}^\ell \bigl(a_i N_{H_i}b_i\cdot m\bigr) \in \sum_{i=1}^\ell a_iM^{H_i}. \qedhere
  \]
\end{proof}

The following proposition shows that the exponent bound is optimal, therefore
justifying Definition~\ref{def:optiden}.
\begin{proposition}
  Let~$M=\Z[G]$ be the left regular~$\Z[G]$-module, and let~$\Hsg$ be a set of
  nontrivial subgroups of~$G$ such that~$d(\Hsg)>0$.
  Let~$N$ be the~$\Z[G]$-submodule generated by~$\sum_{H\in\Hsg}M^H$.
  Then the exponent of~$M/N$ equals~$d(\Hsg)$.
\end{proposition}
\begin{proof}
  By Proposition~\ref{prop:meta}, the exponent divides~$d(\Hsg)$.

  Let~$H\in\Hsg$. Then the set of elements of~$M$ of the form~$N_Hg$, where~$g$
  ranges over a set of representatives of~$G/H$, forms a $\Z$-basis of~$M^H$.
  The~$\Z[G]$-module generated by~$M^H$
  is therefore the two-sided ideal generated by~$N_H$.
  Putting all~$H$ together, we see that the $\Z[G]$-submodule~$N$ equals the
  two-sided ideal generated by the~$N_H$ for~$H\in \Hsg$.
  In particular, the order of the image of~$1$ in the quotient~$M/N$
  equals~$d(\Hsg)$, proving the proposition.
\end{proof}

We will now apply Proposition~\ref{prop:meta} to both the additive and the
multiplicative $\Z[G]$-modules attached to~$K$.

\subsection{Additive structure}

Consider $M = \order_K$, the ring of integers of the number field~$K$. For
every $H \leq G$ we have $M^H = \order_{K^H}$,
where $K^H$ is the fixed field of~$H$.
Thus from Proposition~\ref{prop:meta} we obtain the following statement.
Recall that an order~$\order$ of $K$ is defined to be \emph{$p$-maximal} if
$[\order_K : \order]$ is not divisible by $p$.

\begin{corollary}\label{cor:relmaxorder}
  \hfill
  \begin{enumerate}
    \item
      If $G$ admits a scalar norm relation of the form~(\ref{eq:normrel}), then
      the exponent of the quotient
      \[
        \OO_K / (\OO_{K^{H_1}} + \dotsb + \OO_{K^{H_\ell}})
      \]
      is finite and divides~$d$.
      In particular,
      the ring of integers~$\OO_K$ is generated, as an abelian group, by
      the~$\OO_{K^{H_i}}$ together with any order that is~$p$-maximal at
      all primes~$p\mid d$.
    \item
      If $G$ admits a norm relation of the form~(\ref{eq:normrelgen}), then
      the exponent of the quotient
      \[
        \OO_K / (a_1\OO_{K^{H_1}} + \dotsb + a_\ell\OO_{K^{H_\ell}})
      \]
      is finite and divides~$d$.
      In particular,
      the ring of integers~$\OO_K$ is generated, as a $\Z[G]$-module, by
      the~$\OO_{K^{H_i}}$ together with any order that is~$p$-maximal at
      all primes~$p\mid d$.
  \end{enumerate}
\end{corollary}

\subsection{Multiplicative structure}
The group~$K^\times$ is naturally a~$\Z[G]$-module, of which we will
consider various submodules as follows.
Let $S$ be a $G$-stable set of non-zero prime ideals of $\order_K$.
Recall that
\[ \order_{K,S}^\times = \{ x \in K^\times \, \mid \, v_\p(x) = 0 \text{ for all $\p \not\in S$}\} \]
is the group of $S$-units of $K$.
Let~$L$ be a subfield of $K$; we define the $S$-units of $L$ as
$\order_{L, S'}^\times$ where
$S' = \{ L \cap \p \, \mid \, \p \in S \}$.
The multiplicative group $M = \order_{K, S}^\times$ is a $\Z[G]$-submodule
of~$K^\times$ and for $H \leq G$ we have
$M^H = \order_{K^H, S}^\times$.

Recall that for a finitely generated subgroup $V \subseteq K^\times$ and $d \in
\Z_{>0}$, the \textit{$d$-saturation of~$V$} is the smallest group $W \subseteq K^\times$ such that $V \subseteq W$ and
$K^\times/W$ is $d$-torsion-free.
Similarly, the \textit{saturation of $V$} is the smallest group $W \subseteq K^\times$ such that $V \subseteq W$ and $K^\times/W$ is torsion-free.
The group $V$ is called \textit{$d$-saturated} (resp. \textit{saturated}) if $V$ is equal to its $d$-saturation (resp. saturation).

Note that the~$S$-unit group of $K$ is saturated in~$K^\times$, i.e. every element having a
nonzero power that is an~$S$-unit of $K$ is itself an~$S$-unit of $K$.
Applying Proposition~\ref{prop:meta} to this situation yields the following.

\begin{corollary}\label{cor:relSunits}
  \hfill
  \begin{enumerate}
    \item
      If $G$ admits a scalar norm relation of the form~(\ref{eq:normrel}), then
      the exponent of the quotient
      \[
        \OO_{K, S}^\times / \OO_{K^{H_1}, S}^\times \dotsm \OO_{K^{H_\ell}, S}^\times
      \]
      is finite and divides~$d$.
      In particular,
      the group $\OO_{K, S}^\times$ of $S$-units of~$K$ equals the
      $d$-saturation of $\OO_{K^{H_1}, S} \dotsm \OO_{K^{H_\ell}, S}^\times$.
    \item\label{item:relSunits-generel}
      If $G$ admits a norm relation of the form~(\ref{eq:normrelgen}), then
      the exponent of the quotient
      \[
        \OO_{K, S}^\times / (\OO_{K^{H_1}, S}^{\times})^{a_1} \dotsm
        (\OO_{K^{H_\ell}, S}^\times)^{a_\ell}
      \]
      is finite and divides~$d$.
      In particular,
      the group $\OO_{K, S}^\times$ of $S$-units of~$K$ equals the
      $d$-saturation of the~$\Z[G]$-module generated by
      $(\OO_{K^{H_1}, S}^{\times}) \dotsm (\OO_{K^{H_\ell}, S}^\times)$.
  \end{enumerate}
\end{corollary}

In view of the previous result, one might ask to what extent the relations
between the invariants of~$K$ and of its subfields force a norm
relation.
A positive result in this direction was obtained by Artin
in~\cite{Artin1948} for scalar norm relations, which easily extends
to norm relations as follows.

\begin{proposition}\label{prop:artin}
  The group $G$ admits a norm relation~$d = \sum_{i=1}^\ell a_i N_{H_i} b_i$
  if and only if for all $x \in K^\times$ we have
  $x^d = \prod_{1 \leq i \leq \ell} \norm_{K/K_i}(x^{b_i})^{a_i}$.
\end{proposition}

\begin{proof}
  Consider the element $\sigma = \sum_{i=1}^\ell a_i N_{H_i}b_i - d \in \Z[G]$
  and assume that the equality in $K^\times$ holds.
  Thus $x^\sigma = 1$ for all $x \in K^\times$, that is,
  $\sigma \in \Ann_{\Z[G]}(K^\times)$.
  Since $\Ann_{\Z[G]}(K^\times) = 0$ by \cite[Theorem 5]{Artin1948}, the result
  follows.
\end{proof}

Concerning the structure of the $S$-units, we have the following partial
converse to Corollary~\ref{cor:relSunits}.
We will not use it in this work, but it answers a natural question: in a normal
extension, is the existence of a norm relation necessary to be able to use
subfields to recover the group of $S$-units?

\begin{proposition}\label{prop:conv}
  Let~$K/F$ be a finite normal extension of number fields with Galois group~$G$, and let~$\Hsg$
  be a set of nontrivial subgroups of~$G$.
  Let~$S$ be a finite $G$-stable set of prime ideals of~$K$.
  Assume that at least one of the following holds:
  \begin{itemize}
    \item $F$ is not totally real,
    \item there is a real place of~$F$ that splits completely in~$K$, or
    \item there is a prime ideal~$\pg$ of~$F$ that splits completely in~$K$ and
      such that the primes above~$\pg$ are in~$S$.
  \end{itemize}
  If the $\Z[G]$-submodule of~$\OO_{K,S}^\times$ generated by the~$\OO_{K^H,S}^\times$
  for~$H\in\Hsg$ has finite index, then~$G$ admits a norm relation with respect
  to~$\Hsg$.
\end{proposition}

\begin{proof}
  Under the hypotheses of the proposition, the~$\Q[G]$-module $\OO_{K,S}^\times
  \otimes \Q$ contains a copy of the regular module modulo the trivial
  module, and is generated as a~$\Q[G]$-module by the union of its fixed points
  under the subgroups~$H\in\Hsg$. Now apply
  Proposition~\ref{prop:generel}~(\ref{item:repsQ}) and the fact that for
  every~$H\in\Hsg$ and every simple~$\Q[G]$-module~$V$, the module~$V$ is
  generated by~$V^H$ if and only if~$V^H\neq 0$.
\end{proof}

\subsection{Class group structure}

Let $\Cl(K)$ be the class group of $K$, which is the quotient of the fractional ideals modulo the principal fractional ideals of $K$.
While $\Cl(K)$ is again a $\Z[G]$-module it is not true that $\Cl(K^H) =
\Cl(K)^H$ (in general the natural map $\Cl(K^H) \to \Cl(K)$ is not even injective).

\begin{proposition}\label{prop:apl_normrel}
  Assume that the group~$G$ admits a norm relation~(\ref{eq:normrelgen}).
  Define the maps
  \[
    \Phi\colon \Cl(K) \to \bigoplus_{i=1}^\ell \Cl(K^{H_i}), \ [\ag] \mapsto
    \left(\norm_{K/K^{H_i}}(\ag^{b_i})\right)_i
  \]
  and
  \[
    \Psi \colon \bigoplus_{i=1}^\ell \Cl(K^{H_i}) \longrightarrow \Cl(K),\ \left([\ag_i]\right)_i \longmapsto \prod_i [\ag_i\order_{K}]^{a_i}.
  \]
   Let $R = \Z[\frac 1 d]$.
   Then the map
   \[
     \Phi\otimes R \colon \Cl(K)\otimes R \to \bigoplus_{i=1}^\ell \Cl(K^{H_i})\otimes R
   \]
   is injective, i.e. an isomorphism onto its image, and
   the map
   \[
     \Psi\otimes R \colon \bigoplus_{i=1}^\ell \Cl(K^{H_i})\otimes  R
     \longrightarrow \Cl(K)\otimes  R
   \]
   is surjective.
   The image~$(\Phi\otimes R)(\Cl(K)\otimes  R)$ is a direct summand
   of the group $\bigoplus_{i=1}^\ell \Cl(K^{H_i})\otimes R$, and the
   group~$\Cl(K)/\Cl(K)[d]$ is isomorphic to a subgroup of~$\bigoplus_{i=1}^\ell
   \Cl(K^{H_i})$.
\end{proposition}

\begin{proof}
  The relation~(\ref{eq:normrelgen}) shows that~$\Psi\circ\Phi\colon \Cl(K) \to \Cl(K)$ is the map
  \[
    \Psi\circ\Phi\colon [\ag] \to \prod_i [\ag]^{a_i N_{H_i} b_i} = [\ag]^d,
  \]
  i.e.~$\Psi\circ\Phi = d\Id$. Since~$d$ is invertible in~$R$, this implies
  that~$(\Psi\circ\Phi)\otimes R$ is invertible, and therefore
  that~$\Phi\otimes R$ is injective and~$\Psi\otimes R$ is surjective as claimed.
  Let~$A = \bigoplus_{i=1}^\ell \Cl(K^{H_i})\otimes R$ and~$e =
  d^{-1}(\Phi\circ\Psi)\otimes R \colon A \to A$; then~$e$ is an idempotent, so that~$A =
  eA\oplus (1-e)A$, and we have~$eA = \Phi(\Cl(K)\otimes R)$ by surjectivity
  of~$\Psi\otimes R$.
  Finally, we have a surjection~$\Cl(K)/\ker \Phi \to
  \Cl(K)/\Cl(K)[d]$, and $\Phi$ induces an injection~$\Cl(K)/\ker \Phi \to
  \bigoplus_{i=1}^\ell \Cl(K^{H_i})$, proving that~$\Cl(K)/\Cl(K)[d]$ is a
  subquotient of~$\bigoplus_{i=1}^\ell \Cl(K^{H_i})$. Since every quotient of a
  finite abelian group~$B$ is also isomorphic to a subgroup of~$B$, this proves
  the proposition.
\end{proof}

\subsection{Analytic structure}

For the sake of completeness, we also mention the following classical consequence for the analytic structure of $K$.
For Brauer relations (and therefore also for scalar norm relations), we have the following well known implications for equalities of zeta functions.

\begin{proposition}\label{prop:brauerzeta}
  Suppose~$G$ admits a useful Brauer relation, written in the form
  \[
    a_1 \Ind_{G/1}(\triv_1) = \sum_{1\neq H\le G} a_H \Ind_{G/H}(\triv_H)
  \]
  with~$a_H\in\Z$ and~$a_1>0$.
  Then the following equality of zeta functions holds:
  \[
    \zeta_K(s)^{a_1} = \prod_{1\neq H\le G}\zeta_{K^H}(s)^{a_H}.
  \]
\end{proposition}

\begin{proof}
  See~\cite[Theorem 73]{Frohlich1993}.
\end{proof}

\section{Computational problems in number fields}\label{sec:algo}

We now describe algorithms for solving various computational problems in number fields that exploit the
subfield structure as described in Section~\ref{sec:applications}.
Let $K/F$ be a normal extension of algebraic number fields with Galois group $G$.

\subsection{Construction of relations}\label{subsec:constrrel}

We begin by explaining how to find norm relations.
First note that if $G$ is abelian, then one can use
Theorem~\ref{thm:optimalabelianrels} to write down scalar norm relations directly.
In the general case, let us assume that $\Hsg$ is a set of subgroups.
Considering the $\Q$-subspace
$W = \langle N_H \mid H \in \Hsg \rangle_\Q \subseteq \Q[G]$,
there exists a scalar norm relation if and only if $1 \in W$.
Thus we can find scalar norm relations by using linear algebra over $\Q$.
Similarly, when looking for a scalar norm relation with a specific denominator $d$,
we can check whether $d \in \langle N_H \mid H \in \Hsg\rangle_\Z$ using linear
algebra over $\Z$.
A similar approach works for norm relations.
In this case one has to consider the $\Q$-subspace
$W = \langle N_H \mid H \in \Hsg \rangle_{\Q[G]}
  = \langle g N_H h \mid H \in \Hsg, g, h \in G \rangle_\Q \subseteq \Q[G]$.
From the basic properties of norm elements, we can obtain a smaller generating set for $W$ by considering only the elements $gN_H h$ with $H \in \mathcal H$,
$g \in G/H$ and $h \in \mathcal{N}(H)\backslash G$ where~$\mathcal{N}(H)$
denotes the normalizer of~$H$ in~$G$.

Note that for Brauer relations, a simple linear algebra based method was described by Bosma and de~Smit in~\cite[Section 3]{Bosma2001}.

\subsection{Computing rings of integers}

Let $K/F$ be a normal extension of algebraic number fields with Galois group $G$.
We assume that $G$ admits a norm relation of denominator $d$ of the form
\[
  d = \sum_{i=1}^\ell a_i N_{H_i} b_i
\]
with $H_i \leq G$, $d \in \Z$, $a_i,b_i \in \Z[G]$.
The classical algorithm for computing the ring of integers
$\order_K$ of $K$, that is, finding a $\Z$-basis of~$\order_K$,
proceeds by enlarging a starting order $\order$
successively until $\order = \order_K$ holds (see~\cite[Section
4]{Lenstra1992}); it requires factoring, at least partially, the discriminant
of~$\order$, which can be hard.
Using Corollary~\ref{cor:relmaxorder}, we may alternatively compute the
ring of integers of $\order_K$ using the rings of integers~$\order_{K^{H_i}}$ as
follows.
\begin{enumerate}
  \item
    For each $1 \leq i \leq \ell$ compute $\order_{K^{H_i}}$ classically (or recursively).
  \item
    Determine a $\Z$-basis of the order $\order$ generated by $a_1
    \order_{K^{H_1}} + \dotsb + a_\ell \order_{K^{H_\ell}}$.
  \item \label{item:pimaximal}
    Return $\order + \order_{p_1} + \dotsb + \order_{p_r}$, where
    $p_1,\dotsc,p_r$ denote the prime divisors of~$d$ and $\order_{p_i}$
    denotes the $p_i$-maximal overorder of $\order$ (which can be computed
    efficiently, see~\cite[Theorem 4.5]{Lenstra1992}).
\end{enumerate}

\begin{remark}
  \hfill
  \begin{enumerate}
    \item In case one has a scalar norm relation, that is, $a_i \in \Z$ and $b_i = 1$
      for all $1 \leq i \leq \ell$,
      according to Corollary~\ref{cor:relmaxorder} one can replace $\order$ by the order
      generated by $\order_{K^{H_1}} + \dotsb + \order_{K^{H_\ell}}$.
      In the general case, one can also replace~$\order$ by the order generated
      by the $\Z[G]$-module generated by~$\order_{K^{H_1}} + \dotsb +
      \order_{K^{H_\ell}}$.
    \item By Theorem~\ref{thm:denrelG}, we may always choose the relation so
      that the~$p_i$ are among the prime divisors
      of~$|G|$, and are therefore small. In particular, no hard factorization is
      needed in Step~(\ref{item:pimaximal}).
  \end{enumerate}
\end{remark}

\subsection{Computing $S$-unit groups}\label{subsec:sunit}

Let $K/F$ be a normal extension of algebraic number fields with Galois group $G$.
We assume that $G$ admits a norm relation of denominator $d$ of the form
\[
  d = \sum_{i=1}^\ell a_i N_{H_i} b_i
\]
with $H_i \leq G$, $d \in \Z$, $a_i,b_i \in \Z[G]$.
Let $S$ be a finite $G$-stable set of non-zero prime ideals of $\order_K$.
Our aim is to describe an algorithm for computing a $\Z$-basis of the $S$-unit
group $\order_{K, S}^\times = \{ x \in K \mid v_\p(x) = 0 \text{ for $\p
\not\in S$}\}$.
Based on Corollary~\ref{cor:relSunits}, this can be accomplished by the
following steps.
\begin{enumerate}
  \item
    For each subfield $K_i = K^{H_i}$ determine a basis of the $S$-unit group $\order_{K_i, S}^\times$.
  \item
    Determine the group $V = (\order_{K_i, S}^\times)^{a_1} \dotsm
    (\order_{K_\ell, S}^\times)^{a_\ell} \subseteq \order_{K, S}^\times$.
  \item
    Compute and return the $d$-saturation of $V$.
\end{enumerate}

\begin{remark}
  In case one has a scalar norm relation, that is, $a_i \in \Z$ and $b_i = 1$
  according to Corollary~\ref{cor:relSunits} one can replace $V$ by
  $\order_{K_i, S}^\times \dotsm \order_{K_\ell, S}^\times$.
\end{remark}

The computations in Step~(1) can be done either using the algorithm of Simon~\cite[\S{}I.1.2]{simon} (see also~\cite[7.4.2]{Cohen2000}) or recursively.
As Step~(2) needs no further explanation, we will now describe the saturation in Step~(3).

\subsubsection*{Saturation of finitely generated multiplicative groups.}
Saturation is a well known technique in computational algebraic number theory,
used for example in the class and unit group computation of number fields
(\cite[Section 5.7]{PohstZassenhaus1989}) or the number field sieve
(\cite{Adleman1991}).

We will discuss this problem in the following generality. We let $V \subseteq K^\times$
be a finitely generated subgroup. For a fixed integer $d \in \Z_{>0}$, we wish
to determine the $d$-saturation $W$ of $V$. Recall that this is by definition the smallest
group $W \subseteq K^\times$ with $V \subseteq W$ and $K^\times/W$ $d$-torsion-free.
To determine whether a multiplicative group is $d$-saturated, the following simple result is
crucial.

\begin{lemma}\label{lem:sat}
  Let $V \subseteq K^\times$ be finitely generated. Then the following hold.
  \begin{enumerate}
    \item
      The $d$-saturation of $V$ contains the $d$-torsion of $K^\times$.
    \item
      The group $V$ is $d$-saturated if and only if $V$ is $\noell$-saturated for all primes $\noell$ dividing $d$.
    \item
      For a prime $\noell$ the group $V$ is not $\noell$-saturated if and only if there exists
      $\alpha \in K^\times \setminus V$ with $\alpha^\noell \in V$.
      In this case $\noell$ divides the index $[\langle V, \alpha \rangle : V]$.
    \item
      Let $\noell$ be a prime and assume that $V$ contains the $\noell$-torsion of $K^\times$.
      Then $V$ is $\noell$-saturated if and only if $V \cap (K^\times)^{\noell} = V^\noell$.
  \end{enumerate}
\end{lemma}

\begin{proof}
  (1): Let $W$ be the $d$-saturation of $V$ and let $\alpha \in K^\times$ with $\alpha^d = 1$.
  As~$K^\times / W$ is $d$-torsion-free, this implies $\alpha \in W$.
  Statements (2) and (3) are trivial.
  
  For (4), let us assume that $V \cap (K^\times)^\noell = V^\noell$.
  Now let $\alpha \in K$ with $\alpha^\noell \in V$.
  Thus we have~$\alpha^\noell \in V \cap (K^\times)^\noell = V^\noell$ and there exists $\beta \in V$
  with $\alpha^\noell = \beta^\noell$. Since $V$ contains the $\noell$-torsion elements of $K^\times$
  this implies $\alpha \in V$.
  Now assume that $V$ is $\noell$-saturated and $\alpha \in V \cap (K^\times)^\noell$. Thus
  there exists $\beta \in K^\times$ with $\beta^\noell = \alpha$. As $V$ is $\noell$-saturated this implies
  $\beta \in V$ and $\alpha = \beta^\noell \in V^\noell$.
\end{proof}

Thus, from now on we will assume that $d = \noell$ is a prime.
As there exists a polynomial time algorithm to determine the irreducible
factors and hence the roots of polynomials over $K$ (see \cite{Lenstra1983}), we may
also assume that $V$ contains the $\noell$-torsion of $K^\times$.

The previous lemma makes it clear that the key to saturation is the computation
of $V \cap (K^\times)^\noell$.
To this end, we want to detect global $\noell$-th powers by using local information.
This is used for example in the class and unit group computation of number
fields (\cite[Section 5.7]{PohstZassenhaus1989}) or the number field sieve
(\cite{Adleman1991}).
The building block is the following special case of the Grunwald--Wang theorem,
see~\cite[Chapter~X]{ArtinTate2009} or~\cite[Chapter IX,
\S{}1]{NeukirchSchmidtWingberg2008}. Note that since the exponent is prime,
there is no obstruction to the Hasse principle for $\noell$-th powers.
In the following, for a non-zero prime ideal $\p$ of $\order_K$ we will denote by
$K_\p$ the $\p$-adic completion of $K$, by $v_\p$ the $\p$-adic valuation and by $k_\p = \order_K/\p \cong \order_{K_\p}/\p \order_{K_\p}$
the residue field at $\p$.

\begin{theorem}[Grunwald--Wang]\label{thm:gw}
  For every finite set $S$ of primes of~$K$, the canonical map
  \[ K^\times / (K^\times)^\noell \longrightarrow \prod_{\p \not \in S}
  K_\p^{\times}/(K_\p^\times)^\noell \]
  is injective.
\end{theorem}

To detect local powers, we make use of the following well-known statement.

\begin{proposition}\label{prop:gw1}
  Let $d \in \Z_{>0}$.
  Assume that $\p$ is a non-zero prime ideal of $\mathcal O_K$ with $d \not\in \p$ and let $\varpi \in K$ be a local uniformizer at $\p$.
  Then the map
  \[
    K_\p^\times/(K_\p^\times)^d \longrightarrow \Z/d\Z \times
  k_\p^\times/(k_\p^\times)^d, \quad \bar x \longmapsto (\overline{v}, \overline{x\varpi^{-v}})\text{ where $v = v_\p(x)$},
  \]
  is an isomorphism. In particular if $V \subseteq K^\times$ is a subgroup with $v_\p(\alpha)$ divisible by~$d$ for all $\alpha \in V$, then
  \[ \ker(V/V^{d} \to K_\p/(K_\p^\times)^d) = \ker(V/V^{d}
  \longrightarrow k_\p^{\times}/(k_\p^\times)^d).\]
\end{proposition}

\begin{proof}
  This follows from the properties of unit groups in local fields (see~\cite[Chapter II, \S{}5]{Neukirch}) as follows:
  The map $K_\p^\times \cong \Z \times \order_{K_\p}^\times, \ x \mapsto (v_\p(x), x\varpi^{-v_\p(x)})$ is an isomorphism and
  the group $\order_{K_\p}^\times$ decomposes as $k_\p^\times \times (1 + \p \order_{K_\p}$).
  Denote by $q \in \Z$ the prime lying below $\p$.
As $1 + \p\order_{K_\p} \cong \Z/q^a\Z \times \Z_q^b$ for
integers $a, b \in \Z_{\geq 0}$ and $\gcd(d, q) = 1$, the group $1 + \p
 \order_{K_\p}$ is $d$-divisible and therefore $(1+\p
\order_{K_\p})/(1 + \p \order_{K_\p})^d = 1$.
\end{proof}

For a prime ideal $\p$, we now fix a local uniformizer $\varpi \in K$ at $\p$ and set
\[ \chi_\p \colon K_\p^\times/(K_\p^\times)^d \longrightarrow
\Z/d\Z \times k_\p^\times/(k_\p^\times)^d, \quad \bar x \longmapsto (\overline{v}, \overline{x\varpi^{-v}})\text{ where $v = v_\p(x)$}. \] 
Note that for every subgroup $V \subseteq K^\times$, the map $\chi_\p$
induces a map $V/V^d \to \Z/d\Z \times k^\times_\p/(k_\p^\times)^d$, which by abuse of notation we will also denote by $\chi_\p$.

\begin{proposition}\label{prop:gw2}
  Assume that $S$ is a set of prime ideals of $\mathcal O_K$ and for $d \in \Z_{>0}$ the canonical map
  \[ K^\times / (K^\times)^d \longrightarrow \prod_{\p\not\in S}
  K_\p^{\times}/(K_\p^\times)^d \]
  is injective.
  Then for every subgroup group $V \subseteq K^\times$ we have
  \[ V \cap (K^\times)^d = \bigcap_{\p \not\in S} \ker(V
  \to K_\p^\times /(K_\p^\times)^d), \]
      and if $S$ contains the primes dividing $d$ then
      \[ (V \cap (K^\times)^d)/ V^d = \bigcap_{\p\not\in S}\ker(
      \chi_\p \colon V/V^d \to \Z/d\Z \times k_\p^\times/(k_\p^\times)^d). \]
      In particular these equalities hold for $S = \{ \p \mid d \in \p\}$ and $d = \noell$ a prime.
\end{proposition}

\begin{proof}
  For the first equality note that the assumption implies
  \[ V \cap (K^\times)^{d} = V \cap \bigcap_{\p\not\in S} \ker(K^\times \to K_\p^\times/(K_\p^\times)^{d}) = \bigcap_{\p\not\in S} \ker(V \to K_\p^\times/(K_\p^\times)^{d}). 
  \]
  Thus, using Proposition~\ref{prop:gw1}, we obtain
  \begin{align*} (V \cap (K^\times)^{d})/V^d &= (V \cap \bigcap_{\p\not\in
    S}\ker(K^\times \to K_\p^\times/(K_\p^\times)^{d}))/V^d \\
  &= \bigcap_{\p\not\in S}\ker(V \to K_\p^\times/(K_\p^\times)^{d})/V^d \\
                             &= \bigcap_{\p\not\in S}\ker(V/V^d \to K_\p^\times/(K_\p^\times)^{d}) \\
  &= \bigcap_{\p \not\in S}\ker(\chi_\p \colon V/V^d \to \Z/d\Z \times k_\p^\times/(k_\p^\times)^{d}).
  \end{align*}
  The final statement follows from Theorem~\ref{thm:gw}.
\end{proof}

\begin{corollary}\label{cor:gw}
  Assume that $V \subseteq K^\times$ is finitely generated, and let~$p$ be a
  prime number.
  Then there exists a constant $c_0 \in \R_{>0}$ such that
  \[ (V \cap (K^\times)^\noell)/V^\noell = \bigcap_{p \not\in\p,\, \Nm(\p) \leq c_0} \ker(\chi_\p \colon V/V^\noell \to \Z/p\Z \times k_\p^\times/(k_\p^\times)^{\noell}). \]
\end{corollary}

\begin{proof}
  As $V$ is finitely generated, $V/V^\noell$ is a finite dimensional $\F_\noell$-vector space.
  Thus $V/V^\noell$ is Artinian and the claim follows from Proposition~\ref{prop:gw2}.
\end{proof}

\begin{proposition}\label{prop:gw3}
  Let $c \in \R_{>0}$, $V \subseteq K^\times$ finitely generated, let~$p$ be a
  prime number and let~$m$ be
  the $\F_\noell$-dimension of the intersection
  \[ \bigcap_{p \not\in \p,\, \Nm(\p) \leq c} \ker(\chi_\p \colon V/V^\noell \to
  \Z/p\Z \times k_\p^\times/(k_\p^\times)^{\noell}) \subseteq V/V^\noell,\]
  and let~$\alpha_1,\dots,\alpha_m\in V$ be such
  that~$\overline{\alpha_1},\dotsc,\overline{\alpha_m}$ is an~$\F_\noell$-basis
  of the intersection.
  \begin{enumerate}
    \item
      If $m = 0$, then $V$ is $\noell$-saturated.
    \item
      Assume that $V$ is not $\noell$-saturated.
      Then if $c$ is sufficiently large, there exists $1 \leq i \leq m$ such that $\alpha_i$ is a $\noell$-th power.
    \item
      Assume that $V$ is $\noell$-saturated.
      Then for $c$ sufficiently large we have $m = 0$.
  \end{enumerate}
\end{proposition}

\begin{proof}
  First note that by Lemma~\ref{lem:sat}~(3) the group $V$ is $\noell$-saturated if and only if~$V \cap (K^\times)^\noell = V^\noell$.
  Denote the intersection of the kernels by $W/V^\noell$.
  (1): Follows since $(V \cap (K^\times)^\noell)/V^\noell \subseteq W/V^\noell$.
  (2)~and~(3): For $c = c_0$ as in Corollary~\ref{cor:gw} we have $W/V^\noell = (V \cap (K^\times)^\noell)/V^\noell$. Now apply Lemma~\ref{lem:sat}.
\end{proof}

\begin{algorithm}\label{alg:sat}~
  \begin{itemize}
    \item Input: $V\subset K^\times$ finitely generated.
    \item Output: statement that~$V$ is $\noell$-saturated or an element
      $\alpha$ with $[\langle V, \alpha \rangle : V]$ divisible by $\noell$.
  \end{itemize}
  \begin{enumerate}
    \item
      Let $c \in \R_{>0}$ be any constant.
    \item
      Determine an $\F_\noell$-basis $\overline{\alpha_1},\dotsc,\overline{\alpha_m}$ of
      \[ \bigcap_{p \not\in \p,\, \Nm(\p) \leq c} \ker(V/V^\noell \to \Z/p\Z \times k_\p^\times/(k_\p^\times)^{\noell}).\]
    \item
      If $m = 0$, return that $V$ is $\noell$-saturated.
    \item
      If $m > 0$, test whether one of the elements $\alpha_i$
      is a $\noell$-th power. If there exists $\alpha$ with $\alpha^\noell = \alpha_i$, return $\alpha$.
    \item
      Replace $c$ by $2c$ and go to step~(2).
  \end{enumerate}
\end{algorithm}

\begin{theorem}
  Algorithm~\ref{alg:sat} is correct.
\end{theorem}

\begin{proof}
  Follows immediately from Proposition~\ref{prop:gw3}.
\end{proof}

By iterating Algorithm~\ref{alg:sat} one can compute the $\noell$-saturation of $V$ and by repeating this for every prime $\noell$ dividing $d$, we obtain the $d$-saturation of $V$.

\subsubsection*{Complexity}
In this section we prove polynomial time complexity bounds for the $S$-units and
saturation algorithms.
We first need an effective version of the Grunwald--Wang theorem. Related work
is presented in~\cite{Wang2015}, but the statement we
need is different.

\begin{theorem}[Effective Grunwald--Wang]\label{thm:gruneff}
  Assume GRH.
  Let~$d = \noell^r$ with~$\noell$ prime and~$r\ge 1$.
  Let~$K$ be a number field of degree~$n$, and~$L = K(\zeta_d)$.
  Let~$S$ be a finite set of primes of~$K$, let~$M_S = \prod_{\p\in S}\Nm(\p)$, and
  let~$S_p = S\cup \{\p \mid p\}$.
  Let
  \[
    c_0 = 18d^2\left(2\log|\Delta_K| + 6n\log d + \log M_S\right)^2.
  \]
  Let~$T$ be the set of prime ideals~$\p$ of~$K$ such that
  \begin{itemize}
    \item $\p\notin S_p$,
    \item $\p$ has residue degree~$1$,
    \item $\Nm(\p) \equiv 1\bmod{p}$, and
    \item $\Nm(\p)\le c_0$.
  \end{itemize}
  Let~$\alpha\in K^\times$ be such that all valuations of~$\alpha$ at
  primes~$\p\notin S_p$ are divisible by~$d$ and such that for every~$\p\in T$,
  the image of~$\alpha$ in~$K_\p^\times$ is a~$d$-th power.
  Then~$\alpha \in (L^\times)^d$.
  If in addition~$L/K$ is cyclic, then~$\alpha \in (K^\times)^d$.
\end{theorem}
\begin{proof}
  Note that the degree of~$L/K$ is at most~$\varphi(d)$ and the
  discriminant~$\Delta_L$ of~$L$ satisfies
  \[
    \Delta_L \mid \Delta_K^{\varphi(d)}\Delta_{\Q(\zeta_d)}^n.
  \]

  We first prove that~$\alpha$ is a~$d$-th power in~$L$.
  By contradiction, assume otherwise and let~$\beta$ be a~$d$-th root
  of~$\alpha$ in some extension of~$L$, so that~$L(\beta)/L$ is a
  cyclic extension of degree~$d'\neq 1$ dividing~$d$.
  Let~$\chi$ be a faithful $1$-dimensional character of~$\Gal(L(\beta)/L)$, which we see as a
  ray class group character of~$L$ of some conductor~$\cond$ by class field
  theory. Write~$\cond = \condt \condw$ where~$\condt$ and~$\condw$ are coprime
  and~$\condw$ is supported at primes above~$p$.
  By the assumption on the valuations of~$\alpha$, the extension~$L(\beta)/L$ is
  unramified outside the prime ideals that do not lie above a prime in~$S_p$;
  indeed, locally at every such prime~$\PP$, the extension is generated by a
  $\noell^i$-th root of a unit of~$L_\PP$ for some~$i\le r$.
  Therefore, by~\cite[Proposition 2.5]{Murty1988} applied to~$L(\beta)/L$ and~$\chi$, we
  have
  \[
    \log \Nm(\condw) \le 2n\varphi(d)(\log \noell + \log \varphi(d)) \le
    4nd\log d.
  \]
  In addition, ramification is tame at all primes in~$S$ not above~$p$, so we
  have
  \[
    \Nm(\condt) \le M_S^d.
  \]
  By~\cite[Theorem 4]{Bach1990}, there exists a prime ideal~$\PP$
  of~$L$ that has residue degree~$1$ (so that~$\Nm(\PP)=1\bmod d$),
  does not lie over primes of~$S_p$, such
  that~$\chi(\PP) \neq 1$ and such that
  \[
    \Nm(\PP) \le 18\log^2(\Delta_L^2\Nm(\cond))
    \le 18\left(2d\log|\Delta_K| + 6nd\log d + d\log M_S\right)^2 = c_0.
  \]
  In particular, the prime ideal~$\p = \PP\cap K$ lies in~$T$, so~$\alpha$ is a~$d$-th
  power in~$K_{\p}^\times$, and \emph{a fortiori} in~$L_{\PP}^\times$. This
  implies that~$L(\beta)/L$ is completely split at~$\PP$, contradicting the fact
  that~$\chi(\PP)\neq 1$.
  This proves that~$\alpha\in (L^\times)^d$.

  Now assume that~$L/K$ is cyclic, and let~$L' = K(\zeta_p)$.
  Let~$\beta_1,\dots,\beta_d\in L$ be the $d$-th roots of~$\alpha$, so that we
  have~$L'\subseteq L'(\beta_i) \subseteq L$ for all~$i$. Since~$L/L'$ is a cyclic
  extension of degree a power of~$p$, its intermediate extensions are linearly
  ordered, so that we may choose our numbering so that~$L'(\beta_1) \subseteq
  L'(\beta_i)$ for all~$i$.

  Assume for contradiction that the cyclic extension~$L'(\beta_1)/L'$ is
  nontrivial. Then as above there exists a nontrivial character~$\chi$
  of~$\Gal(L'(\beta_1)/L')$ and a prime ideal~$\PP$ of~$L'$ of degree~$1$ (so
  that~$\Nm(\PP)=1\bmod p$) such
  that~$\chi(\PP)\neq 1$ and~$\p = \PP\cap K$ lies in~$T$.
  By hypothesis, $\alpha$ is a $d$-th power in~$K_\p^\times$, so that there
  exists~$i$ such that~$\beta_i\in K_\p$; since~$L'(\beta_1) \subseteq
  L'(\beta_i)$, this implies~$\beta_1 \in L'_\PP$, i.e. $L'(\beta_1)/L'$ is
  completely split at~$\PP$, contradicting~$\chi(\PP)\neq 1$.
  This proves that the extension~$L'(\beta_1)/L'$ is trivial, i.e.~$\beta_1\in
  L'$.

  We have~$\beta_1^d = \alpha$, so that~$\alpha^{[L':K]} = \Nm_{L'/K}(\alpha)
  = \Nm_{L'/K}(\beta_1)^d$, and therefore~$\alpha^{[L':K]}$ is a $d$-th power
  in~$K$. Since~$[L':K]$ is coprime to~$d$, this implies that~$\alpha$ is a $d$-th
  power in~$K$, as claimed.
\end{proof}

\begin{remark}
  \hfill
  \begin{enumerate}
    \item We did not try to optimize the value~$c_0$, only to obtain an
      explicit value from readily available results in the literature.
    \item For~$\p\notin S_p$, since the valuations of~$\alpha$ at~$\p$ is divisible by~$d$
      and~$\noell\notin\p$, the assumption that~$\alpha$ is a $d$-th power
      in~$K_\p^\times$ is equivalent to the reduction modulo~$\p$
      of~$\alpha\varpi^{-v}$ being a $d$-th power, where~$v=v_\p(\alpha)$
      and~$\varpi\in K$ is such that~$v_\p(\varpi)=1$.
  \end{enumerate}
\end{remark}

\begin{corollary}\label{cor:gruneff}
  Assume GRH.
  There exists a deterministic polynomial time algorithm, that given~$m$
  generators of a subgroup~$V$ of~$K^\times$ and an integer~$d$ that is either~$2$ or
  a power of an odd prime,
  determines $\alpha_1,\dotsc,\alpha_m \in K^\times$ such that
  $\overline{\alpha_1},\dotsc,\overline{\alpha_m}$ generate~$(V \cap
  (K^{\times})^d)/V^d$.
\end{corollary}

\begin{proof}
  Let~$S$ be the smallest set of primes of~$K$ such
  that~$V\subseteq\OO_{K,S}^\times$, and~$M_S = \prod_{\p\in S}N(\p)$. Then~$\log
  M_S$ is polynomial in the size of the input, and can be bounded in polynomial
  time without factoring the given generators of~$V$.
  From Proposition~\ref{prop:gw2} and Theorem~\ref{thm:gruneff} it follows that
  \[ (V \cap (K^{\times})^d)/V^d = \bigcap_{d \not\in \mathfrak p, \Nm(\mathfrak
  p) \leq c_0} \ker(\chi_\p \colon V/V^d \to \Z/d\Z \times k_\p^\times /
  (k_\p^{\times})^d), \]
  where $c_0 = 18d^2\left(2\log|\Delta_K| + 6n\log d + \log M_S\right)^2$.
  Since the number and size of the primes~$\p$ are polynomial in the input, this
  proves the claim.
\end{proof}

We need an extra technical ingredient: a suitably normalized logarithmic
embedding to control the height of $S$-units that appear.
\begin{definition}\label{def:heightlogemb}
  Let~$K$ be a number field with~$r_1$ real embeddings and~$r_2$ pairs of
  conjugate complex embeddings.
  For~$x\in K$, the (logarithmic) \emph{height} of~$x$ is
  \[
    \height(x) = \sum_\sigma \max\Bigl(0, n_\sigma\log|\sigma(x)|\Bigr) +
    \sum_{\p} \max \Bigl(0,(\log N\p)v_\p(x) \Bigr),
  \]
  where~$\sigma$ ranges over the~$r_1+r_2$ conjugacy classes of complex embeddings of~$K$,
  and~$n_\sigma=1$ or~$2$ according to whether~$\sigma$ is real or complex,
  and~$\p$ ranges over nonzero prime ideals of~$K$.

  Let~$S$ be a finite set of nonzero prime ideals of~$K$. The \emph{logarithmic embedding}
  attached to~$S$ is the map~$\Logemb\colon \OO_{K,S}^\times \to
  \R^{r_1+r_2+|S|}$ defined by
  \[
    \Logemb(x) = \Bigl(n_\sigma\log|\sigma(x)|\Bigl)_\sigma \times \Bigl((\log
    N\p)v_\p(x)\Bigl)_\p.
  \]
\end{definition}

\begin{lemma}\label{lem:heightlogemb}
  Use the notations of Definition~\ref{def:heightlogemb}.
  For all~$x\in\OO_{K,S}^\times$, we have
  \[
    \height(x) \le \sqrt{r_1+r_2+|S|} \cdot \left\|\Logemb(x)\right\|_2.
  \]
\end{lemma}
\begin{proof}
  This follows immediately from the inequality between the $L^1$ and $L^2$
  norms.
\end{proof}

As a conclusion to this section, we prove a polynomial time reduction for the
computation of~$S$-units from~$K$ to its subfields in the presence of a
norm relation.

\begin{algorithm}\label{alg:redtosubfields}
  Assume that the finite group~$G$ admits a norm relation with
  respect to a set~$\Hsg$ of subgroups of~$G$.
  \begin{itemize}
    \item Input: a number field~$K$, an injection~$G\to\Aut(K)$, a finite
      $G$-stable set~$S$ of prime ideals of~$K$, and for each~$H\in\Hsg$, a
      $\Z$-basis~$B_H$ of~$\order_{K^H,S}^\times$.
    \item Output: a $\Z$-basis of~$\order_{K,S}^\times$.
  \end{itemize}
  \begin{enumerate}
    \item\label{step:primesG} Let~$p_1=2 < p_2 <\dots < p_k$ be the prime divisors of~$2|G|$.
    \item\label{step:Gaction} Let~$B = \bigcup_{H\in\Hsg}\bigcup_{g\in G}g(B_H)$.
    \item\label{step:defV} Let~$V\subseteq \order_{K,S}^\times$ be the subgroup generated by~$B$.
    \item $v\sto$ the $2$-adic valuation of~$|G|^3$.
    \item $V_1 \sto V$.
    \item\label{step:saturation2} Repeat~$v$ times
      \begin{enumerate}
        \item\label{step:roots2} $V_1 \sto \langle V_1, \sqrt{\alpha_1},\dots,\sqrt{\alpha_m}\rangle$
          where~$\overline{\alpha_1},\dotsc,\overline{\alpha_m}$ is a
          basis of~$(V_1 \cap (K^{\times})^2)/V_1^2$.
        \item\label{step:keepsmall} reduce the basis of~$V_1$ with respect
          to~$B$ in the sense of \cite[Lemma~7.1]{MicciancioGoldwasser}
          in the logarithmic embedding~$\Logemb$.
      \end{enumerate}
    \item\label{step:saturationodd} For~$i=2$ to~$k$
      \begin{enumerate}
        \item $v\sto$ the $p_i$-adic valuation of~$|G|^3$.
        \item\label{step:rootsodd} $V_i\sto$ the
          $p_i$-saturation of~$V$ by taking~$d$-th roots once, where~$d=p_i^v$.
      \end{enumerate}
    \item\label{step:gatherV} $V \leftarrow V_1\cdots V_l$.
    \item\label{step:basisV} Return a basis of~$V$.
  \end{enumerate}
\end{algorithm}

\begin{remark}
  Algorithm~\ref{alg:redtosubfields} never writes down an explicit norm
  relation; it only uses the fact that there exists one.
  Note that for given $\mathcal H$, this can be checked in polynomial time, see Section~\ref{subsec:constrrel}.
\end{remark}

We now prove the bit complexity of~Algorithm~\ref{alg:redtosubfields}.
In order to do so, we use the model of Lenstra~\cite{Lenstra1992} to encode the input of the algorithm.

\begin{theorem}\label{thm:redtosubfields}
  Assume GRH. Let $G$ be a finite group and $\Hsg$ a set of subgroups 
  of~$G$. 
  Assume that there exists a norm relation with respect to~$\Hsg$.
  Then Algorithm~\ref{alg:redtosubfields} is a deterministic polynomial time
  algorithm that, on input of
  \begin{itemize}
    \item a number field $K$,
    \item an injection $G \to \Aut(K)$,
    \item a finite $G$-stable set $S$ of prime ideals of $K$,
    \item for each $H$ in $\Hsg$, a basis of the group of $S$-units of the subfield 
          fixed by $H$,
  \end{itemize}
  returns a $\Z$-basis of the group of $S$-units of $K$.
\end{theorem}

\begin{proof}
  Denote by $n$ the degree of the number field $K$ over $\Q$.
  Note that the height and
  bitsize of elements are bounded relatively to each other by a polynomial
  in the size of the input.
  We will therefore measure the size of various
  quantities in terms of height, without affecting the polynomial time claim of
  the algorithm.

  Let~$\Sigma$ denote the total size of the input.
  By hypothesis there exists a norm relation in~$G$.
  Moreover, by Theorem~\ref{thm:denrelG}, we may assume that the denominator of
  the relation divides~$|G|^3$.

  Step~(\ref{step:primesG}) only requires factoring~$|G| = O(n) = O(\Sigma)$ and
  therefore takes polynomial time.

  After Step~(\ref{step:Gaction}), since the action of automorphisms does not
  change the height of elements, the total size of~$B$ is~$O(n\Sigma)$.

  Note that in Steps~(\ref{step:defV})--(\ref{step:basisV}),
  one can deduce a basis from a generating set of the groups involved
  in polynomial time: the algorithms of~\cite{Ge} provide a basis
  of the relations between the generators, and the Hermite normal form \cite{hafner_HNF} allows us to obtain a basis of the group in polynomial time.

  Consider a saturation step~(\ref{step:roots2}) or~(\ref{step:rootsodd})
  corresponding to taking~$d$-th roots.
  By Corollary~\ref{cor:gruneff} we can determine generators
  $\overline{\alpha_1},\dotsc,\overline{\alpha_m}$ of $(V \cap
  (K^{\times})^d)/V^d$ in polynomial time. Computing the roots themselves also
  takes polynomial time.
  Moreover, in the loop~(\ref{step:saturation2}), Step~(\ref{step:keepsmall})
  together with Lemma~\ref{lem:heightlogemb} make sure that the size of~$V_1$
  stays bounded by a polynomial in~$\Sigma$ independent of the number of steps.
  Therefore the loops~(\ref{step:saturation2}) and~(\ref{step:saturationodd})
  take polynomial time.

  The steps~(\ref{step:gatherV}) and~(\ref{step:basisV}) take polynomial time in
  the data computed at this point.

  The correctness of the algorithm follows from Corollary \ref{cor:relSunits}
  (\ref{item:relSunits-generel}).
\end{proof}

\begin{remark}
  In Theorem~\ref{thm:redtosubfields}, the $S$-units of $K$ and its subfields are represented with respect to an integral basis.
  It is well known that using this, the representation can require exponentially
  large coefficients with respect to the discriminant of the field.
  An alternative approach is to represent the $S$-units of the input as well as the output, that is, of the subfields as
  well as of $\mathcal O^\times_{K, S}$, using compact representations.
  That the statement remains true using compact representations follows from~\cite{Biasse2016}, where it
  is shown that compact representations can be computed in polynomial time.
\end{remark}

\subsection{Computing class groups}\label{sec:algoclgp}

Assume that $K/F$ is a normal extension of number fields with Galois group $G$
that admits a norm relation
\[
  d = \sum_{i=1}^\ell a_i N_{H_i} b_i
\]
with $H_i \leq G$, $d \in \Z$, $a_i,b_i \in \Z[G]$.
We now describe how to use
this to determine the class group of $K$ from the class groups of the subfields $K^{H_i}$.
Let $S$ be a finite set of prime ideals that generates the class group $\Cl(K)$ of $K$.
Assuming the generalized Riemann hypothesis (GRH) we can use for example Bach's bound on the maximal norm of the prime
ideals required to generate $\Cl(K)$ and the set $S = \{ \p \ | \ \Nm(\p) \leq 12 \cdot
\log(\lvert \Delta_K\rvert)^2 \}$ (see~\cite{Bach1990}) or one can compute an ad-hoc set $S$ using the methods of~\cite{Belabas2008} or~\cite{Grenie2018}.

\subsubsection*{Using $S$-units}
Using the algorithm of Section~\ref{subsec:sunit} we can determine a $\Z$-basis of the $S$-unit group $\order_{K, S}^\times$.
Now as in Buchmanns's algorithm~\cite{Bsub}, consider the map
\[ \varphi \colon \order_{K, S}^\times \longrightarrow \Z^{\lvert S \rvert}, \ \alpha \longmapsto (v_\p(\alpha))_{\p \in S}.\]
Then $\Cl(K) \cong \coker(\varphi)$, since the sequence
$\order_{K, S}^\times \xrightarrow{\varphi} \Z^{\lvert S \rvert} \xrightarrow{\psi} \Cl(K) \rightarrow 0$
is exact, where $\psi((v_\p)_{\p \in S}) = [ \prod_{\p \in S} \p^{v_\p}]$. 

\subsubsection*{Direct computation}

We now consider the map
$\Cl(K) \otimes \Z[\frac 1 d] \rightarrow \bigoplus_{i=1}^\ell \Cl(K^{H_i}) \otimes \Z[\frac 1 d]$, $[\ag] \mapsto ([\Nm_{K/K^{H_i}}(\ag^{b_i})]),$
which by Proposition~\ref{prop:apl_normrel} is an isomorphism.
Hence $\Cl(K) \otimes \Z[\frac 1 d] \cong \langle \Phi(\p) \, | \, \p \in S \rangle \otimes \Z[\frac 1 d]$.
In particular, if one is interested only in the $\noell$-part of the class group
for some prime $\noell$ not dividing $d$ or if the denominator~$d$ of the
norm relation is equal to $1$, this provides a second way to determine the
structure of the class group.

\subsection{Class groups of abelian extensions}\label{subsec:clgab}

In this section we describe a Las Vegas algorithm based on the ideas above to compute the
class group of an abelian field. Contrary to Algorithm~\ref{alg:redtosubfields} and its
variants, the algorithm we present here never computes an explicit $d$-th root,
and therefore completely avoids using LLL, and does not need a Bach-type bound to
certify its correctness, making it very fast in practice.
This is possible because we are only asking for
the structure of the class group and not for explicit units, $S$-units or
generators of ideals, which would be computationally harder.

Let~$K/F$ be a normal extension of number fields with abelian Galois group~$G$.

Write~$G\cong C \times Q$ where $C$ is the largest cyclic factor of~$G$.
According to Theorem~\ref{thm:optimalabelianrels}, we will have the following three cases.
\begin{enumerate}
  \item The order $|Q|$ has at least two distinct prime divisors. Write $Q\cong P_1\times \dotsb \times P_k$ with~$P_i$ abelian $p_i$-groups with~$p_i$
    distinct primes.
    This case does not require any saturation, and reduces to computations of
    class groups in various subfields~$K_j/F$ with Galois groups that
    are isomorphic to subgroups of~$C\times P_i$.
  \item The group $Q$ is a nontrivial $p$-group for some prime~$p$. Then we apply the
    methods from Section~\ref{subsec:sunit}
, using a relation with denominator a power of~$p$.
    This case requires $p$-saturation, and reduces to computations of class
    groups and units in various subfields~$K_i/F$ with Galois groups that are
    isomorphic to subgroups of~$C$.
  \item We have $Q=1$: then the norm relation method does not apply, so we simply use
    Buchmann's algorithm~\cite{Bsub} (or any other algorithm that can compute
    the class group and units).
\end{enumerate}

The algorithms corresponding to cases~(1) and~(2) are the following.

\begin{algorithm}\label{alg:ab1}
Assume $|Q|$ has at least two distinct prime divisors. Write $Q\cong
  P_1\times P_2\times \dots \times P_k$ with~$P_i$ abelian $p_i$-groups (case 1
  above).
  \begin{itemize}
    \item Input: $K/F$ with Galois group~$G$.
    \item Output: the class group~$\Cl(K)$.
  \end{itemize}
\begin{enumerate}
  \item Use Theorem~\ref{thm:optimalabelianrels} to write a norm relation with
    denominator~$1$, involving a collection~$(H_j)$ of subgroups of~$G$ such
    that each~$G/H_j$ is isomorphic to a subgroup of some~$C\times P_i$.
  \item Compute the subfields~$K_j = K^{H_j}$.
  \item Compute the class groups of the subfields~$K_j$, and for each subfield,
    a set~$S_j$ of prime ideals that generates the class group.
  \item Let~$S = \bigcup_j \{\pg\order_K \mid \pg\in S_j \}$.
  \item Compute the image~$\clcomp$ of~$S$ in~$\bigoplus_j \Cl(K_j)$ under the
    map~$\Phi$ of Proposition~\ref{prop:apl_normrel}.
  \item Return~$\clcomp$.
\end{enumerate}
\end{algorithm}

\begin{remark}
  The ideals in $S$ are not necessarily prime, but we only use the property that
  their images generate the class group~$\Cl(K)$.
\end{remark}

\begin{proposition}
  Algorithm~\ref{alg:ab1} correctly computes the class group of~$K$.
\end{proposition}
\begin{proof}
  By the surjectivity part of Proposition~\ref{prop:apl_normrel}, $S$ generates
  the class group of~$K$.
  By the injectivity part of Proposition~\ref{prop:apl_normrel}, $\clcomp$ is
  isomorphic to the class group of~$K$.
  This prove the correctness of the algorithm.
\end{proof}

\begin{algorithm}\label{alg:ab2}
  Assume~$Q$ is a nontrivial $p$-group (case~2 above).
  \begin{itemize}
    \item Input: $K/F$ with Galois group~$G$.
    \item Output: the class group~$\Cl(K)$.
  \end{itemize}
\begin{enumerate}
  \item Use Theorem~\ref{thm:optimalabelianrels} to write a norm relation with
    denominator~$d$ a power of~$p$, involving a collection~$(H_i)$ of subgroups
    of~$G$ such that each~$G/H_i$ is isomorphic to a subgroup of~$C$.
  \item Compute the subfields~$K_i = K^{H_i}$, and their class groups and
    units.
  \item Compute the coprime-to-$p$ part of~$\Cl(K)$ as follows.
    \begin{enumerate}
      \item For each~$K_i$, compute a set~$S_i$ of prime ideals that generates
        the coprime-to-$p$ part of the class group.
      \item Let~$S' = \bigcup_i \{\pg\order_K \mid \pg\in S_i \}$.
      \item Compute the image~$\clcomp_{p'}$ of~$S'$ in~$\bigoplus_i \Cl(K_i)_{p'}$ under the
    map~$\Phi$ of Proposition~\ref{prop:apl_normrel}.
    \end{enumerate}
  \item Let~$h_{p'} = \lvert\clcomp_{p'}\rvert$.
  \item Compute the roots of unity~$W$ in~$K$.
  \item\label{step:brauer} By seeing the relation as a Brauer relation using
    Proposition~\ref{prop:brauernormrel}, compute~$\hR_K = h_K\Reg_K$ from
    the same quantity in the subfields using Proposition~\ref{prop:brauerzeta}
    and the analytic class number formula.
  \item Compute the subgroup~$U_0$ of $\order_K^\times$ generated
    by~$\bigcup_i \order_{K_i}^\times$, and the regulator~$R_0$ of~$U_0$;
    let~$r_0$ be the number of generators of~$U_0$.
  \item Initialize a set~$T$ of prime ideals~$\pg$ such that~$N(\pg)\equiv 1\bmod d$.

    \emph{The primes in $T$~will be used to detect $d$-th powers.}
  \item Initialize a set~$S_\Q$ of prime numbers, and compute the
    set~$S$ of all prime ideals of~$K$ above the primes in~$S_\Q$.

    \emph{We hope that~$S$ will generate the $p$-part of the class group.}
  \item Compute the $p$-part of~$\Cl(K)$ as follows.
    \begin{enumerate}
      \item\label{step:saturationloop} Compute the subgroup~$U_S$ of~$\order_{K,S}^\times$ generated
        by~$\bigcup_i \order_{K_i,S}^\times$; let~$r$ be the number of
        generators of~$U_S$.
      \item Compute the map
        \[
          f \colon \Z^r
          \longrightarrow U_S \longrightarrow \Z^S \oplus \bigoplus_{\p\in T}\F_\p^\times
          \longrightarrow (\Z/d\Z)^{|S|+|T|}
        \]
        given by the
        valuations at prime ideals in~$S$ and discrete logarithms
        in~$\F_\pg^\times$ for~$\pg\in T$.
      \item Compute~$V_S = \ker f$ and~$V_0 = \ker \bigl(f\colon \Z^{r_0} \to U_0 \to
        (\Z/d\Z)^T\bigr)$.
      \item Compute~$u = \left|V_0/\bigl(V_0^d\cdot (W\cap V_0)\bigr)\right|$.
      \item\label{step:valimage} Compute the subgroup~$V$ generated by the image
        of~$U_S \to \Z^S$ and $\frac{1}{d}$~times the image of~$V_S \to U_S \to
        \Z^S$.
      \item\label{step:pclassgroup} Compute the~$p$-part~$\clcomp_p$ of~$\Z^{S}/V$.
      \item Let~$h_p = \lvert\clcomp_p\rvert$.
      \item If~$h_{p'}R_0h_p/u > \hR_K/2$, then return~$\clcomp_{p'}\times\clcomp_p$;
        otherwise increase $T$ and $S_\Q$ and go back
        to~(\ref{step:saturationloop}).
    \end{enumerate}
\end{enumerate}
\end{algorithm}

\begin{remark}
  As before, the ideals in $S'$ are not necessarily prime.
  In our implementation, which is restricted to~$F = \Q$,
  we initialize~$T$ with~$|T| = 10\ +$ unit rank of~$K$, and we increase it by
  adding random prime ideals of norm~$\approx (d\log|\Delta_K|)^2$;
  we initialize~$S_\Q$ with $S_\Q = \emptyset$, and we increase it by adding
  random primes of norm~$\approx (\log|\Delta_K|)^2$ that split completely
  in~$K$.
\end{remark}

\begin{proposition}\label{prop:algab2}
  If Algorithm~\ref{alg:ab2} terminates, then its output is correct.
\end{proposition}
\begin{proof}
  By the surjectivity part of Proposition~\ref{prop:apl_normrel}, $S'$ generates
  the coprime-to-$p$ part of the class group of~$K$.
  By the injectivity part of Proposition~\ref{prop:apl_normrel}, $\clcomp_{p'}$
  is isomorphic to the coprime-to-$p$ part of the class group of~$K$; in
  particular~$h_{p'} = \lvert\Cl_{p'}\rvert$.
  At Step~\ref{step:saturationloop}, $U_S$ satisfies~$\order_{K,S}^\times/U_S$
  has exponent dividing $d$ by Corollary~\ref{cor:relSunits}. Therefore, the
  subgroup~$V$ of~$\Z^S$ computed at Step~\ref{step:valimage} contains the image
  of~$\order_{K,S}^\times$; in particular~$h_p$ is a divisor of the $p$-part of
  the subgroup of the class group generated by~$S$, and equals the $p$-class
  number if and only if~$S$ generates the $p$-part of~$\Cl(K)_p$ and~$V$ equals
  the image of~$\order_{K,S}^\times$.
  In addition~$\Reg_K$ is a $p$-power multiple of~$R_0/u$ by
  Corollary~\ref{cor:relSunits}.
  Therefore, if the algorithm terminates, then~$h_{p'}R_0h_p/u = h_K\Reg_K =
  \hR_K$, the group~$\clcomp_p$ is isomorphic to~$\Cl(K)_p$, and the output is
  correct.
\end{proof}

\begin{remark}
  It may happen that Algorithm~\ref{alg:ab2} does not terminate if~$K$ has an
  obstruction to the Hasse principle for~$d$-th powers. These obstructions are
  characterized by the Grunwald--Wang theorem, and can only happen if~$d\ge 8$
  is a power of~$2$.
  We currently do not know how to avoid this without computing an actual $d$-th
  root in~$K$.
\end{remark}

\begin{remark}\label{rem:simpleralgo}
  If one only needs the structure of the class group of~$K$ as an abstract abelian
  group, as opposed to having explicit ideal classes as generators, one may
  replace the use of Proposition~\ref{prop:apl_normrel} in
  Algorithms~\ref{alg:ab1} and~\ref{alg:ab2} by~\cite[Proposition~2.2 and
  Corollary 1.4]{Boltje}.
\end{remark}

\subsection{Unconditional computations}

Computations of class groups are typically done under GRH and later certified by
a different algorithm. The algorithms of Section~\ref{subsec:clgab} are
oblivious to the method used to compute the information in the subfields: if the
class group, regulator, unit and $S$-unit groups of the subfields are correct,
then so is the output of Algorithms~\ref{alg:ab1} and~\ref{alg:ab2}. However, it
can take a very long time to fully certify the information from the subfields.
In this section, we describe a method to certify the class group structure
assuming only partial information on the subfields.
This is easy in Algorithm~\ref{alg:ab1}: since the class group is computed via
Proposition~\ref{prop:apl_normrel} or Remark~\ref{rem:simpleralgo}, it is
correct as soon as the class groups of the subfields are correct.
Throughout this section, we will refer to the notations in
Algorithm~\ref{alg:ab2}, such as~$U_0$, $V_S$, etc.
It will be convenient to have, for various objects, a separate
notation for the correct one and the one that was computed; we will denote the
latter with a tilde: for instance~$\order_K^\times$ is the unit group
of~$\order_K$ and~$\widetilde{\order_K^\times}$ is the subgroup that was
computed.

\begin{proposition}\label{prop:certif}
  Let~$K/F$ be a normal extension of number fields with Galois group~$G$ as in
  Algorithm~\ref{alg:ab2}.
  Let
  \[
    c_0\Ind_{G/1}{\triv_{1}} = \sum_{i} c_i \Ind_{G/{H_i}}(\triv_{H_i})
  \]
  be the Brauer relation used in Step~\ref{step:brauer} with~$c_i\in\Z$ and $c_0>0$.
  Assume the following:
  \begin{enumerate}
    \item\label{ass:cl} for all~$i$, the computed class group~$\widetilde{\Cl(K_i)}$ is correct;
    \item\label{ass:tu} for all~$i$, the computed group of roots of unity in~$K_i$ is correct, and the
      computed group of roots of unity in~$K$ is correct;
    \item\label{ass:fu} for all~$i$, the computed unit group~$\widetilde{\order_{K_i}^\times}$
      is a subgroup of~$\order_{K_i}^\times$ of finite index at
      most~$B_i \ge 1$ and of index coprime to~$p$;
    \item\label{ass:fsu} for all~$i$, the computed $S$-unit group~$\widetilde{\order_{K_i,S}^\times}$
      is a subgroup of~$\order_{K_i,S}^\times$ of finite index coprime to~$p$;
    \item\label{ass:ratio} at the end of the algorithm, we have
      \[
        \left|
        \Bigl(\frac{\widetilde{h_{p'}}\widetilde{R_0}\widetilde{h_p}/\tilde{u}}{\widetilde{\hR_K}}
        \Bigr)^{c_0}
        - 1 \right| < d^{-(|S|+r_0)c_0}\prod_i B_i^{-|c_i|}.
      \]
  \end{enumerate}
  Then the class group output by Algorithm~\ref{alg:ab2} is correct.
\end{proposition}
\begin{proof}
  Since the~coprime-to-$p$-part of the class group is computed via
  Proposition~\ref{prop:apl_normrel} or Remark~\ref{rem:simpleralgo}, it is
  correct by Assumption~(\ref{ass:cl}). We now focus on the~$p$-part.
  Let~$U_0$ be the subgroup of~$\order_K^\times$ generated by
  the~$\order_{K_i}^\times$, and define~$U_S$ similarly for~$S$-units.
  Let~$R_0$ be the regulator of~$U_0$. Let~$h_p$ and~$h_{p'}$ be the $p$-part
  and coprime-to-$p$-part of the class number of~$K$. Let~$u =
  [\order_K^\times/W : U_0/W]$ where~$W$ is the group of
  roots of unity in~$K$. Let~$\hR_K = h_K\Reg_K$. We have
  \begin{equation}\label{eq:hRsat}
    h_{p'}R_0h_p/u = \hR_K.
  \end{equation}
  Let
  \[
    \rho_1 =
    \frac{\widetilde{h_{p'}}\widetilde{R_0}\widetilde{h_p}/\tilde{u}}{\widetilde{\hR_K}}
  \]
  be the quantity appearing in Assumption~(\ref{ass:ratio}), which we expect to be~$1$.
  Since the coprime-to-$p$-part of the class group is correct, we
  have~$\widetilde{h_{p'}} = h_{p'}$.
  Let~$w_i$ be the number of roots of unity in~$K_i$ and~$w_0$ the number of
  roots of unity in~$K$.
  By the analytic class number formula and Proposition~\ref{prop:brauerzeta}, we
  have
  \[
    \left(\frac{\hR_K}{w_0}\right)^{c_0} =
    \prod_i\left(\frac{h_{K_i}\Reg_{K_i}}{w_i}\right)^{c_i}.
  \]
  By Assumptions~(\ref{ass:cl}) and~(\ref{ass:tu}) and Step~\ref{step:brauer} we have
  \[
    \left(\frac{\widetilde{\hR_K}}{w_0}\right)^{c_0} =
    \prod_i\left(\frac{h_{K_i}\widetilde{\Reg_{K_i}}}{w_i}\right)^{c_i}.
  \]
  The quotient of these two equations gives two expressions for a quantity that
  we denote by~$\rho_2$:
  \[
    \rho_2 = \left(\frac{\widetilde{\hR_K}}{\hR_K}\right)^{c_0} =
    \prod_i \left(\frac{\widetilde{\Reg_{K_i}}}{\Reg_{K_i}}\right)^{c_i}.
  \]
  Then by Assumption~(\ref{ass:fu}), $\rho_2$ is a
  positive rational number whose numerator and denominator are bounded
  by~$\prod_i B_i^{|c_i|}$ and such that~$v_p(\rho_2) = 0$.
  Let
  \[
    \rho_3 = \frac{\widetilde{R_0}\widetilde{h_p}/\tilde{u}}{R_0h_p/u},
  \]
  which is a positive rational number since~$\widetilde{U_0}$ is a finite index
  subgroup of~$U_0$. By Assumption~(\ref{ass:fu}), the ratio
  $
    \frac{\widetilde{R_0}}{R_0}
  $
  is an integer coprime to~$p$.
  Both~$h_p/\widetilde{h_p}$ and~$u/\tilde{u}$ are rational numbers that are
  powers of~$p$.
  Moreover, by Assumption~(\ref{ass:fsu}), they are integers (see also proof of
  Proposition~\ref{prop:algab2}): $\widetilde{h_p}$ can be
  strictly smaller than~$h_p$ only if~$S$ does not generate the $p$-part of
  the class group or if~$T$ is insufficient to correctly detect the $d$-th
  powers in~$\order_{K,S}^\times$, yielding extra elements
  in the computed group of principal ideals;
  similarly~$\tilde{u}$ can be strictly smaller than~$u$ only if~$T$ is
  insufficient to correctly detect the $d$-th powers in~$\order_K^\times$.
  In particular, if~$h_p = \widetilde{h_p}$ and~$u=\tilde{u}$ then the computed
  $p$-part of the class group is correct.
  The ratio
  $
    \rho_4 = \dfrac{h_p/u}{\widetilde{h_p}/\tilde{u}}
  $
  is therefore also an integer and a power of~$p$; moreover, we have~$\rho_4 \le
  d^{|S|+r_0}$ by construction.
  By equation~(\ref{eq:hRsat}) we have
  \[
    \rho_1^{c_0} = \rho_2^{-1}\rho_3^{c_0}.
  \]
  Putting together the previous observations, we obtain that~$\rho_1^{c_0}$ is a
  rational number with denominator at most
  \[
    B = d^{(|S|+r_0)c_0}\prod_i B_i^{|c_i|}
  \]
  and we have~$v_p(\rho_1) = v_p(\rho_3) = -v_p(\rho_4)$.
  Finally, by Assumption~(\ref{ass:ratio}), we have
  $|\rho_1^{c_0}-1| < B^{-1}$, so that~$\rho_1^{c_0}=1$.
  This proves that~$v_p(\rho_4) = -v_p(\rho_1) = 0$ and therefore~$\rho_4=1$,
  proving that the computed $p$-part of the class group of~$K$ is correct.
\end{proof}

\begin{theorem}\label{thm:unconditional-cyclo}
  The class numbers and class groups in Tables~\ref{tab:certified-cyclo}
  and~\ref{tab:certified-cyclo-complete} are correct.
\end{theorem}
\begin{proof}
  We apply our~\textsc{Pari/GP} implementation of Algorithms~\ref{alg:ab1}
  and~\ref{alg:ab2} with GRH-conditional computations in the subfields. Then we
  verify the hypotheses of Proposition~\ref{prop:certif}:
  Assumption~(\ref{ass:tu}) is automatically guaranteed by the~\textsc{Pari/GP}
  functions (as this can be done in polynomial time);
  Assumptions~(\ref{ass:cl}), (\ref{ass:fu}) and~(\ref{ass:fsu}) are checked
  with a modified version of the~\textsc{Pari/GP} function \texttt{bnfcertify};
  Assumption~(\ref{ass:ratio}) is checked by computing the relevant quantity up
  to sufficiently high accuracy. This proves that the class group structures are
  correct. We compute the minus part of the class number by the analytic class
  number formula~\cite[Theorem~4.17]{Washington}, and we deduce the plus part of the class
  number from it.
\end{proof}

\section{Numerical examples}\label{sec:examples}

We have implemented the algorithms from Section~\ref{sec:algo} for computing
$S$-unit and class groups in~\textsc{Hecke}~\cite{Hecke} and
\textsc{Pari/GP}~\cite{PARI2}. More precisely, we implemented in \textsc{Pari/GP} the
algorithms of Section~\ref{subsec:clgab} (with the variant of
Remark~\ref{rem:simpleralgo}) that are restricted to abelian groups, and in
\textsc{Hecke} the algorithms of Sections~\ref{subsec:constrrel}, \ref{subsec:sunit}
and~\ref{sec:algoclgp} that can handle arbitrary groups.
The \textsc{Pari/GP} implementation is available at~\cite{abbnf}.
In this section we report on some numerical
examples obtained using these implementations.
All the computations performed in this section assume GRH.

We begin with a non-abelian example taken from the database of Kl\"uners and Malle~(\cite{Kluners2001}).

\begin{example}\label{ex:hecke}
      The splitting field $K$ of the irreducible polynomial $f = x^{10} + x^8 - 4x^2 + 4 \in \Q[x]$
      has Galois group $C_2 \times A_5$ and discriminant $2^{210} \cdot  17^{80} \approx 10^{161}$.
      Our implementation in \textsc{Hecke} shows that the class group of~$K$ is
      trivial.
      As we use the algorithm using $S$-units of Section~\ref{sec:algoclgp}, we also obtain generators for the unit group and
      obtain the value
      \[ 589229345997607340093151477907958.37876... \]
      for the regulator of $K$.
      The algorithm uses a relation of $C_2 \times A_5$ with denominator~$1$
      involving subfields of degree at most $60$.
      The computation takes $6$ hours on a single core machine.
\end{example}

The remaining examples all concern number fields with abelian Galois group.
Here, we use the Algorithms of Section~\ref{subsec:clgab}.

\begin{example}
  Let~$K = \Q(\zeta_{216})$, which has Galois group over~$\Q$ isomorphic to~$C_{18}\times
  C_2 \times C_2$, degree~$72$ and discriminant~$\approx 10^{129}$.
  Our \textsc{Pari/GP} implementation computes in~$6$ seconds that the class group
  of~$K$ is isomorphic to~$C_{1714617} \cong C_{3^2}\times C_{19} \times C_{37}
  \times C_{271}$. \textsc{Pari/GP} computes the same result in~$15$ minutes,
  and \textsc{magma} in $5$ hours.
  Our algorithm uses a relation with denominator~$4$, and starts by computing
  the class group and units of~$8$ subfields of degree up to~$18$. It then
  starts with~$S=\emptyset$, which turns out to be a generating set for
  the~$2$-class group of~$K$. The algorithm therefore only needs to compute a
  single kernel modulo~$4$ to determine the correct class group at~$2$; the
  units of the subfields generate a subgroup of index~$2^{11}$
  of~$\OO_K^\times$.
\end{example}

\begin{example}\label{ex:parigp}
  Let~$K = \Q(\zeta_{6552})$ which has Galois group over~$\Q$ isomorphic
  to~$C_{12} \times C_6^2 \times C_2^2$, degree~$1728$ and
  discriminant~$2^{3456} \cdot 3^{2592} \cdot 7^{1440} \cdot 13^{1584} \approx
  10^{5258}$. Our \textsc{Pari/GP} implementation computes in~$4.2$ hours that
  the class group of~$K$ is isomorphic to
  \begin{eqnarray*}
    & &
    C_{e} \times C_{123903346647650690244963498417984355147621683400320} \\
    & & \times C_{5827775875747592369293192320} \times C_{2098524198141572423040} \\
    & & \times C_{33847164486154393920} \times C_{7383876252480} \times C_{101148989760}^2 \\
    & & \times C_{50574494880}^2 \times C_{276363360}^5 \times C_{7469280}^2
    \times C_{3734640}^8 \times C_{196560}^2 \\
    & & \times C_{98280} \times C_{32760}^4 \times C_{6552}^{26}
    \times C_{3276}^2 \times C_{252} \times C_{84}^3 \times C_{12}^{29}
    \times C_{6}^8 \times C_2^{11}
  \end{eqnarray*}
  where
  \begin{eqnarray*}
    e &=&
        349380029706737059104248223565319692883897548638392856641627842 \\
    & & 662891732318286799812329621077189995594165744361859090214550165 \\
    & & 734555558870589729949013150675968232635365760,
  \end{eqnarray*}
  and that~$h_{6552}^+ = 70695077806080 = 2^{24}\cdot 3^3 \cdot 5 \cdot 7^4
  \cdot 13$.
  Our algorithm uses a relation with denominator~$1$ involving~$62$ subfields of
  degree at most~$192$. The computations in those subfields recursively uses
  relations with denominators supported at a single primes ($2$ or~$3$),
  involving a total of~$672$ subfields of degree at most~$12$.
\end{example}

\bibliographystyle{plain}
\bibliography{biblio}

\begin{thebibliography}{10}

\bibitem{Adleman1991}
L.~M. Adleman.
\newblock Factoring numbers using singular integers.
\newblock In {\em Proceedings of the Twenty-Third Annual ACM Symposium on
  Theory of Computing}, STOC ’91, page 64–71, New York, NY, USA, 1991.
  Association for Computing Machinery.

\bibitem{AokiFukuda}
M.~Aoki and T.~Fukuda.
\newblock An algorithm for computing {$p$}-class groups of abelian number
  fields.
\newblock In {\em Algorithmic number theory}, volume 4076 of {\em Lecture Notes
  in Comput. Sci.}, pages 56--71. Springer, Berlin, 2006.

\bibitem{Artin1948}
E.~Artin.
\newblock Linear mappings and the existence of a normal basis.
\newblock In {\em Studies and {E}ssays {P}resented to {R}. {C}ourant on his
  60th {B}irthday, {J}anuary 8, 1948}, pages 1--5. Interscience Publishers,
  Inc., New York, 1948.

\bibitem{ArtinTate2009}
E.~Artin and J.~Tate.
\newblock {\em Class field theory}.
\newblock AMS Chelsea Publishing, Providence, RI, 2009.
\newblock Reprinted with corrections from the 1967 original.

\bibitem{Bach1990}
E.~Bach.
\newblock Explicit bounds for primality testing and related problems.
\newblock {\em Math. Comp.}, 55(191):355--380, 1990.

\bibitem{Bartel2012}
A.~Bartel.
\newblock On {B}rauer-{K}uroda type relations of {$S$}-class numbers in
  dihedral extensions.
\newblock {\em J. Reine Angew. Math.}, 668:211--244, 2012.

\bibitem{Bartel2013}
A.~Bartel and B.~de~Smit.
\newblock Index formulae for integral {G}alois modules.
\newblock {\em J. Lond. Math. Soc. (2)}, 88(3):845--859, 2013.

\bibitem{Bartel2014}
A.~Bartel and T.~Dokchitser.
\newblock Brauer relations in finite groups {II}: {Q}uasi-elementary groups of
  order {$p^aq$}.
\newblock {\em J. Group Theory}, 17(3):381--393, 2014.

\bibitem{Bartel2015}
A.~Bartel and T.~Dokchitser.
\newblock Brauer relations in finite groups.
\newblock {\em J. Eur. Math. Soc.}, 17(10):2473--2512, 2015.

\bibitem{BBVLV}
J.~Bauch, D.~Bernstein, H.~de~Valence, T.~Lange, and C.~van Vredendaal.
\newblock Short generators without quantum computers: the case of
  multiquadratics.
\newblock In {\em Advances in cryptology---{EUROCRYPT} 2017. {P}art {I}},
  volume 10210 of {\em Lecture Notes in Comput. Sci.}, pages 27--59. Springer,
  Cham, 2017.

\bibitem{Belabas2008}
K.~Belabas, F.~Diaz~y Diaz, and E.~Friedman.
\newblock Small generators of the ideal class group.
\newblock {\em Math. Comp.}, 77(262):1185--1197, 2008.

\bibitem{Biasse2016}
J.-F. Biasse and F.~Song.
\newblock Efficient quantum algorithms for computing class groups and solving
  the principal ideal problem in arbitrary degree number fields.
\newblock In {\em Proceedings of the {T}wenty-{S}eventh {A}nnual {ACM}-{SIAM}
  {S}ymposium on {D}iscrete {A}lgorithms}, pages 893--902. ACM, New York, 2016.

\bibitem{Biasse2019}
J.-F. Biasse and C.~van Vredendaal.
\newblock Fast multiquadratic {$S$}-unit computation and application to the
  calculation of class groups.
\newblock In {\em Proceedings of the {T}hirteenth {A}lgorithmic {N}umber
  {T}heory {S}ymposium}, volume~2 of {\em Open Book Ser.}, pages 103--118.
  Math. Sci. Publ., Berkeley, CA, 2019.

\bibitem{Boltje}
R.~Boltje.
\newblock Class group relations from {B}urnside ring idempotents.
\newblock {\em J. Number Theory}, 66(2):291--305, 1997.

\bibitem{Bosma2001}
W.~Bosma and B.~de~Smit.
\newblock Class number relations from a computational point of view.
\newblock In {\em Computational algebra and number theory (Milwaukee, WI,
  1996)}, volume 31, no. 1-2, pages 97--112. Elsevier, 2001.

\bibitem{Brauer1951}
R.~Brauer.
\newblock Beziehungen zwischen {K}lassenzahlen von {T}eilk\"{o}rpern eines
  galoisschen {K}\"{o}rpers.
\newblock {\em Math. Nachr.}, 4:158--174, 1951.

\bibitem{Bsub}
J.~Buchmann.
\newblock A subexponential algorithm for the determination of class groups and
  regulators of algebraic number fields.
\newblock In {\em S\'{e}minaire de Th\'{e}orie des Nombres}, pages 27--41,
  Paris, 1988-89.

\bibitem{Cohen2000}
H.~Cohen.
\newblock {\em Advanced topics in computational number theory}, volume 193 of
  {\em Graduate Texts in Mathematics}.
\newblock Springer-Verlag, New York, 2000.

\bibitem{Curtis1990}
C.~W. Curtis and I.~Reiner.
\newblock {\em Methods of representation theory. {V}ol. {I}}.
\newblock Wiley Classics Library. John Wiley \& Sons, Inc., New York, 1990.

\bibitem{Dirichlet1942}
G.~L. Dirichlet.
\newblock Recherches sur les formes quadratiques \`a co\"{e}fficients et \`a
  ind\'{e}termin\'{e}es complexes. {P}remi\`ere partie.
\newblock {\em J. Reine Angew. Math.}, 24:291--371, 1842.

\bibitem{Hecke}
C.~Fieker, W.~Hart, T.~Hofmann, and F.~Johansson.
\newblock Nemo/hecke: Computer algebra and number theory packages for the julia
  programming language.
\newblock In M.~Burr, C.~Yap, and M.~Safey~El Din, editors, {\em Proceedings of
  the 2017 {ACM} on International Symposium on Symbolic and Algebraic
  Computation, {ISSAC} 2017, Kaiserslautern, Germany, July 25-28, 2017}, pages
  157--164. {ACM}, 2017.

\bibitem{Frohlich1993}
A.~Fr\"{o}hlich and M.~J. Taylor.
\newblock {\em Algebraic number theory}, volume~27 of {\em Cambridge Studies in
  Advanced Mathematics}.
\newblock Cambridge University Press, Cambridge, 1993.

\bibitem{Frohlich1971}
A.~Fr\"{o}hlich and C.~T.~C. Wall.
\newblock Equivariant {B}rauergroups in algebraic number theory.
\newblock In {\em Colloque de {T}h\'{e}orie des {N}ombres ({U}niv. de
  {B}ordeaux, {B}ordeaux, 1969)}, pages 91--96. Bull. Soc. Math. France,
  M\'{e}m No. 25. Soc. Math. France, 1971.

\bibitem{Funakura1978}
T.~Funakura.
\newblock On {A}rtin theorem of induced characters.
\newblock {\em Comment. Math. Univ. St. Paul.}, 27(1):51--58, 1978/79.

\bibitem{Ge}
G.~Ge.
\newblock {\em Algorithms related to multiplicative representations of
  algebraic numbers}.
\newblock PhD thesis, University of California, Berkeley, 1993.

\bibitem{Grenie2018}
L.~Greni\'{e} and G.~Molteni.
\newblock Explicit bounds for generators of the class group.
\newblock {\em Math. Comp.}, 87(313):2483--2511, 2018.

\bibitem{hafner_HNF}
J.L. Hafner and K.S. McCurley.
\newblock Asymptotically fast triangulation of matrices over rings.
\newblock In {\em SODA '90: Proceedings of the first annual ACM-SIAM symposium
  on Discrete algorithms}, pages 194--200, Philadelphia, PA, USA, 1990. Society
  for Industrial and Applied Mathematics.

\bibitem{Kani1989}
E.~Kani and M.~Rosen.
\newblock Idempotent relations and factors of {J}acobians.
\newblock {\em Math. Ann.}, 284(2):307--327, 1989.

\bibitem{Kani1994}
E.~Kani and M.~Rosen.
\newblock Idempotent relations among arithmetic invariants attached to number
  fields and algebraic varieties.
\newblock {\em J. Number Theory}, 46(2):230--254, 1994.

\bibitem{Kluners2001}
J.~Kl\"{u}ners and G.~Malle.
\newblock A database for field extensions of the rationals.
\newblock {\em LMS J. Comput. Math.}, 4:182--196, 2001.

\bibitem{Kontogeorgis2008}
A.~Kontogeorgis.
\newblock Actions of {G}alois groups on invariants of number fields.
\newblock {\em J. Number Theory}, 128(6):1587--1601, 2008.

\bibitem{Kuroda1950}
S.~Kuroda.
\newblock \"{U}ber die {K}lassenzahlen algebraischer {Z}ahlk\"{o}rper.
\newblock {\em Nagoya Math. J.}, 1:1--10, 1950.

\bibitem{Lenstra1983}
A.~K. Lenstra.
\newblock Factoring polynomials over algebraic number fields.
\newblock In {\em Computer algebra ({L}ondon, 1983)}, volume 162 of {\em
  Lecture Notes in Comput. Sci.}, pages 245--254. Springer, Berlin, 1983.

\bibitem{Lenstra1992}
H.~W. Lenstra, Jr.
\newblock Algorithms in algebraic number theory.
\newblock {\em Bull. Amer. Math. Soc. (N.S.)}, 26(2):211--244, 1992.

\bibitem{LPS}
A.~Lesavourey, T.~Plantard, and W.~Susilo.
\newblock Short {P}rincipal {I}deal {P}roblem in multicubic fields.
\newblock {\em J. Math. Cryptol.}, 14(1):359--392, 2020.

\bibitem{MicciancioGoldwasser}
D.~Micciancio and S.~Goldwasser.
\newblock {\em Complexity of lattice problems}, volume 671 of {\em The Kluwer
  International Series in Engineering and Computer Science}.
\newblock Kluwer Academic Publishers, Boston, MA, 2002.

\bibitem{Miller}
J.~C. Miller.
\newblock Class numbers of real cyclotomic fields of composite conductor.
\newblock {\em LMS J. Comput. Math.}, 17(suppl. A):404--417, 2014.

\bibitem{Murty1988}
M.~R. Murty, V.~K. Murty, and N.~Saradha.
\newblock Modular forms and the {C}hebotarev density theorem.
\newblock {\em Amer. J. Math.}, 110(2):253--281, 1988.

\bibitem{Neukirch}
J.~Neukirch.
\newblock {\em Algebraic number theory}.
\newblock Comprehensive Studies in Mathematics. Springer-Verlag, 1999.
\newblock ISBN 3-540-65399-6.

\bibitem{NeukirchSchmidtWingberg2008}
J.~Neukirch, A.~Schmidt, and K.~Wingberg.
\newblock {\em Cohomology of number fields}, volume 323 of {\em Grundlehren der
  Mathematischen Wissenschaften}.
\newblock Springer-Verlag, Berlin, second edition, 2008.

\bibitem{abbnf}
A.~Page.
\newblock abelianbnf, 2020.
\newblock Available at \url{https://hal.inria.fr/hal-02961482}.

\bibitem{Park1990}
H.~Park.
\newblock {\em Idempotent relations and the conjecture of {B}irch and
  {S}winnerton-{D}yer}.
\newblock PhD thesis, Brown University, 1990.

\bibitem{Park1996}
H.~Park.
\newblock Relations among {S}hafarevich-{T}ate groups.
\newblock {\em S\={u}rikaisekikenky\={u}sho K\={o}ky\={u}roku}, 998:117--125,
  1997.

\bibitem{Parry1977}
C.~J. Parry.
\newblock Class number formulae for bicubic fields.
\newblock {\em Illinois J. Math.}, 21(1):148--163, 1977.

\bibitem{PohstZassenhaus1989}
M.~Pohst and H.~Zassenhaus.
\newblock {\em Algorithmic algebraic number theory}, volume~30 of {\em
  Encyclopedia of Mathematics and its Applications}.
\newblock Cambridge University Press, Cambridge, 1989.

\bibitem{Rehm1975}
H.~P. Rehm.
\newblock \"{U}ber die gruppentheoretische {S}truktur der {R}elationen zwischen
  {R}elativnormabbildungen in endlichen {G}aloisschen
  {K}\"{o}rpererweiterungen.
\newblock {\em J. Number Theory}, 7:49--70, 1975.

\bibitem{Reiner}
I.~Reiner.
\newblock {\em Maximal orders}, volume~28 of {\em London Mathematical Society
  Monographs. New Series}.
\newblock The Clarendon Press, Oxford University Press, Oxford, 2003.
\newblock Corrected reprint of the 1975 original, With a foreword by M. J.
  Taylor.

\bibitem{simon}
D.~Simon.
\newblock {\em {\'{E}}quations dans les corps de nombres et discriminants
  minimaux}.
\newblock PhD thesis, Universit{\'e} Bordeaux {I}, 1998.

\bibitem{PARI2}
{The PARI~Group}, Univ. Bordeaux.
\newblock {\em {PARI/GP version \texttt{2.13.1}}}, 2021.
\newblock available from \url{http://pari.math.u-bordeaux.fr/}.

\bibitem{Wada1966}
H.~Wada.
\newblock On the class number and the unit group of certain algebraic number
  fields.
\newblock {\em J. Fac. Sci. Univ. Tokyo Sect. I}, 13:201--209 (1966), 1966.

\bibitem{Wall2013}
C.~T.~C. Wall.
\newblock On the structure of finite groups with periodic cohomology.
\newblock In {\em Lie groups: structure, actions, and representations}, volume
  306 of {\em Progr. Math.}, pages 381--413. Birkh\"{a}user/Springer, New York,
  2013.

\bibitem{Walter1979b}
C.~D. Walter.
\newblock Kuroda's class number relation.
\newblock {\em Acta Arith.}, 35(1):41--51, 1979.

\bibitem{Wang2015}
S.~Wang.
\newblock Grunwald-{W}ang theorem, an effective version.
\newblock {\em Sci. China Math.}, 58(8):1589--1606, 2015.

\bibitem{Washington}
L.~C. Washington.
\newblock {\em Introduction to cyclotomic fields}, volume~83 of {\em Graduate
  Texts in Mathematics}.
\newblock Springer-Verlag, New York, 1982.

\bibitem{Wolf1972}
J.~A. Wolf.
\newblock {\em Spaces of constant curvature}.
\newblock AMS Chelsea Publishing, Providence, RI, sixth edition, 2011.

\bibitem{Yu2003}
H.~Yu.
\newblock Idempotent relations and the conjecture of {B}irch and
  {S}winnerton-{D}yer.
\newblock {\em Math. Ann.}, 327(1):67--78, 2003.

\end{thebibliography}

\newpage

\appendix

\section{Class groups of large cyclotomic fields}

\begin{table}[h]
  \caption{Class groups of cyclotomic fields
  $\Q(\zeta_n)$}\label{tab:certified-cyclo-complete}
  $n$ conductor, $\varphi(n)$ degree, $h^+$ plus part of class number, $\Cl$ list of cyclic
  factors of the class group, with multiplicities denoted by exponents.
  \hspace*{-1.3cm}\begin{tabular}[t]{cccl}
    \hline
    $n$ & $\varphi(n)$ & $h^+$ & $\Cl$ \\
    \hline

    $255$ & $128$ & $1$ & $[198604775280, 85]$ \\
    $272$ & $128$ & $2$ & $[38972318856432, 48, 16^2]$ \\
    $320$ & $128$ & $1$ & $[2679767564295, 51, 17^2]$ \\
    $340$ & $128$ & $1$ & $[189394569680, 80^2]$ \\
    $408$ & $128$ & $2$ & $[383350665840, 48, 16^2, 2]$ \\
    $480$ & $128$ & $1$ & $[208430880, 1680, 84, 21]$ \\
    $273$ & $144$ & $1$ & $[112080696, 11544, 8^2, 4^3, 2^2]$ \\
    $315$ & $144$ & $1$ & $[58787820, 606060, 28, 4]$ \\
    $364$ & $144$ & $1$ & $[1212120, 4680, 1560^2, 78, 2]$ \\
    $456$ & $144$ & $1$ & $[4536718103988, 1197, 171, 19]$ \\
    $468$ & $144$ & $1$ & $[130450320, 102960, 468, 117, 3^2]$ \\
    $504$ & $144$ & $4$ & $[39312, 13104, 252^3, 126, 2^3]$ \\
    $520$ & $192$ & $4$ & $[3008481840, 808080, 21840, 80, 16, 8^3, 4^5, 2^5]$ \\
    $560$ & $192$ & $1$ & $[334945469854703482320, 302640, 60, 3^2]$ \\
    $624$ & $192$ & $1$ & $[5435580272293080, 79560, 79560, 195, 65, 5]$ \\
    $720$ & $192$ & $1$ & $[145097043589680, 261908020920, 390, 15]$ \\
    $780$ & $192$ & $1$ & $[3256946160, 208^3, 104, 8^5, 4^4, 2^4]$ \\
    $840$ & $192$ & $1$ & $[43161155222640, 404040, 1560, 780, 2^2]$ \\

    $455$ & $288$ & $1$ & $[2552186819979516720, 39582182640, 161616, 7696, 3848,$ \\
      &&&
      $52^4, 4, 2^4]$
      \\

    $585$ & $288$ & $1$ & $[20973979753397601680869341964560, 7405922160, 10920,$ \\
      &&&
      $5460, 52, 2^2]$
      \\

    $728$ & $288$ & $20$ & $[127601328297438646560, 241506720, 622440, 4680^3,
      312^2,$ \\
      &&&
      $24, 12^2, 6^3, 2^3]$
      \\

    $936$ & $288$ & $16$ & $[380292996258447608175840, 4957112160, 6552^2,
      3276^3,$ \\
      &&&
      $156^2, 12^2]$
      \\

    $1008$ & $288$ & $16$ & $[13191784813235120785056, 2542176, 6552^2, 3276^4,
      156^2,$ \\
      &&&
      $52, 2^2]$
      \\

    $1092$ & $288$ & $1$ & $[1873781327428920, 23030280, 10920^2, 2184, 312^2,
      104^2, 8^4,$ \\
      &&&
      $4^5, 2^6]$
      \\

    $1260$ & $288$ & $1$ & $[302534211670334280, 8464152747960, 32760, 10920,
      2184,$ \\
      &&&
      $168^2, 8, 4^2, 2^4]$
      \\

    $1560$ & $384$ & $8$ & $[397816187272397451623520, 98585760, 1616160, 43680,$ \\
      &&&
      $21840, 1040^4, 208, 16^3, 8^7, 4^{10}, 2^{10}]$
      \\

    $1680$ & $384$ & $1$ & $[20781830230484468879265337592380391572320,156767520,$ \\
      &&&
      $22395360, 605280, 3120, 1560^3, 40, 4^2, 2]$
      \\

    $2520$ & $576$ & $208$ &
      $[197258297436388965346129923485464425191214071085120,$
      \\
      &&&
      $112607968430745627840, 366503875920, 2424240, 65520,$
      \\
      &&&
      $32760^5, 6552, 2184^3, 312, 8^4, 4^5, 2^{14}]$
      \\
    \hline
  \end{tabular}
\end{table}

\end{document}